\documentclass[12pt,a4paper]{amsart}
\usepackage{amssymb,amscd}
\usepackage{amsfonts}
\usepackage[top=35mm, bottom=35mm, left=30mm, right=30mm]{geometry}
\usepackage[colorlinks=true,citecolor=blue]{hyperref}
\usepackage{mathptmx}
\usepackage{eucal}
\usepackage{graphicx}
\usepackage{mathrsfs}
\usepackage{amssymb}
\usepackage{amsmath}
\usepackage{amsthm}
\usepackage{xcolor}
\usepackage[pagewise]{lineno}\nolinenumbers

\newtheorem{theorem}{Theorem}[section]
\newtheorem{proposition}[theorem]{Proposition}
\newtheorem{lemma}[theorem]{Lemma}

\newtheorem{corollary}[theorem]{Corollary}

\newtheorem*{open question}{Open Question}

\newtheorem{remark}[theorem]{Remark}

\theoremstyle{definition}
\newtheorem{definition}[theorem]{Definition}

\numberwithin{equation}{section}

\makeatletter

\newcommand{\Rmnum}[1]{\expandafter\@slowromancap\romannumeral #1@}
\makeatother

\begin{document}
\title{Ergodic measures of intermediate entropies for $\mathbb{Z}^{d}$-action}
\author{Yage Liu, Ercai Chen and Xiaoyao Zhou*}
\address
{1.School of Mathematical Sciences,\small Ministry of Education Key Laboratory of NSLSCS\\	\small  Nanjing Normal University, Nanjing 210023, Jiangsu, P.R.China}
\email{liuyage16@163.com }
\email{ecchen@njnu.edu.cn}
\email{zhouxiaoyaodeyouxian@126.com}
\renewcommand{\thefootnote}{}
\footnotetext{*Corresponding author:zhouxiaoyaodeyouxian@126.com}
\renewcommand{\thefootnote}{\arabic{footnote}}
\begin{abstract}
For dynamical systems satisfying the approximate $\mathbb{Z}^{d}$ or $\mathbb{Z}_+^{d}$-product property and asymptotically entropy expansiveness, we establish a precise description of the structure of their space of invariant measures. In particular, we prove that the set of ergodic measures with any given intermediate entropy is generic in certain natural subspaces. As a consequence, this result confirms Katok's conjecture on the existence of ergodic measures with arbitrary intermediate entropy for such systems.
\end{abstract}
\maketitle
\section{Introduction and main results}
\subsection{Introduction}
Whether positive topological entropy forces a rich structure on the space of invariant measures is a long‑standing problem in dynamical systems. Parry raised a specific version of this question: must a strictly ergodic (i.e. minimal and uniquely ergodic) system necessarily have zero topological entropy? The answer turns out to be negative, and numerous $C^0$ counterexamples have been constructed (e.g.~\cite{BCL07,GW94,HK67}). For smooth systems, however, a positive answer appears plausible: a conjecture attributed to Herman. His reasoning is that positive topological entropy implies the existence of nonzero Lyapunov exponents, from which one can extract hyperbolic behavior. In a seminal contribution \cite{K80}, Katok proved that for $C^{1+\alpha}$ diffeomorphisms on a two‑dimensional manifold, positive topological entropy indeed yields horseshoes. As a consequence, such systems possess ergodic measures with every intermediate metric entropy. Katok also conjectured that this property should hold for any sufficiently regular smooth system in any dimension.

In this paper, let $L = \mathbb{Z}^d$ or $\mathbb{Z}_+^d$ with $d\geq 1$. By a \emph{topological dynamical system} we mean a triple $(X, d, \mathcal{T})$, where $(X,d)$ is a compact metric space and $\mathcal{T}$ is a continuous $L$-action on $X$. We denote by $\mathcal{M}(X)$ the space of Borel probability measures on $X$, by $\mathcal{M}(X, \mathcal{T})$ the subspace of $\mathcal{T}$-invariant measures, and by $\mathcal{M}_e(X,\mathcal{T})$ the set of ergodic measures. Let $h(\mathcal{T}) = h(X, \mathcal{T})$ denote the topological entropy, and for $\mu \in \mathcal{M}(X, \mathcal{T})$, let $h_\mu(\mathcal{T})$ denote the entropy of $\mu$. \textbf{Katok's conjecture} asserts that for any $C^2$ diffeomorphism $T$ on a compact Riemannian manifold $X$, the collection of ergodic entropies
$$ \mathbb{H}(X, T) := \{ h_\mu(T) : \mu \in \mathcal{M}_e(X,T) \}$$
contains $[0, h(T))$.

In this work, we establish that for a continuous $L$-action $\mathcal{T}$, the combination of asymptotically entropy expansiveness and the approximate $L$-product property yields the intermediate entropy property, which is precisely the statement of Katok's conjecture. Our main theorem can be stated as follows.
\begin{theorem}\label{thm 1.1}
Let $(X,d,\mathcal{T})$ be an asymptotically entropy expansive system with the approximate $L$-product property. Then 
	\begin{align*}
	\mathbb{H}(X,\mathcal{T})=[0,h(\mathcal{T})].       
	\end{align*}
\end{theorem}
Significant strides have been made toward resolving Katok's conjecture in a variety of dynamical contexts. In early work, Sun established the conjecture for certain skew product systems \cite{S10a, S10b} and later extended it to linear toral automorphisms, including both hyperbolic and non-hyperbolic cases \cite{S12}. A major breakthrough came from Quas and Soo \cite{QS16}, who demonstrated that any asymptotically entropy expansive system satisfying both the almost weak specification property and the small boundary property is universal, thereby yielding the intermediate entropy property. This result was subsequently refined by Burguet \cite{B20}, who showed that systems with the specification property are topologically universal or almost Borel universal; in particular, any aperiodic subshift whose entropy is lower can be embedded either topologically or in an almost Borel sense. Chandgotia and Meyerovitch \cite{CM21} further advanced this line of inquiry by proving that a generic homeomorphism on a compact manifold of dimension at least two can model any ergodic transformation; they also demonstrated that the non-uniform specification property implies almost Borel universality. Separately, Guan, Sun and Wu \cite{GSW17} established the almost weak specification property for certain homogeneous systems, thereby confirming Katok's conjecture within the universality framework. It is worth emphasizing that the approximate product property alone does not suffice for universality. The ideas introduced by Quas and Soo also played a pivotal role in the work of Huang, Xu and Xu \cite{HXX21}, where they proved the conjecture for affine transformations of nilmanifolds.

Parallel to these developments, alternative approaches have been pursued. Ures \cite{U12} verified Katok's conjecture for a class of partially hyperbolic diffeomorphisms with one-dimensional center. Yang and Zhang \cite{YZ20} examined a broad family of robustly non-hyperbolic transitive diffeomorphisms and showed that every ergodic measure can be approximated, both in the weak-$*$ topology and in entropy, by hyperbolic sets, thereby confirming the conjecture. Konieczny, Kupsa and Kwietniak \cite{KKK18} proved the conjecture for shift spaces with a safe symbol (which includes all hereditary shifts) by demonstrating that the set of ergodic invariant measures is arcwise connected under the $d$-bar metric. Li and Oprocha \cite{LO18} showed that in any topologically transitive system with the shadowing property, ergodic measures supported on odometers are dense in the space of invariant measures. Furthermore, when the entropy map is upper semi-continuous, ergodic measures with entropy exactly $c$ are generic among invariant measures of entropy at least $c$, thereby confirming the conjecture. More recently, Sun \cite{S21} developed a novel approach based on uniqueness of equilibrium states, proving the conjecture for a class of Mañé systems. In subsequent work \cite{S25}, Sun provided a further verification of Katok's conjecture by characterizing a fine structure of the invariant measure space: ergodic measures of intermediate entropies and those of intermediate pressures are generic in suitable subspaces.

The hypothesis of Theorem 1.1 includes the approximate $L$-product property, a concept originally due to Pfister and Sullivan \cite{PS05}. This condition ranks among the mildest of the so-called \emph{specification-like properties} \cite{KLO16}. Such tracing properties are intimately connected with hyperbolicity and have proven indispensable in the analysis of smooth dynamical systems. The theme traces back to Bowen's pioneering work \cite{B71}, where the classical specification property was introduced to investigate periodic points and invariant measures for Axiom A diffeomorphisms. In the intervening decades, a host of variants have emerged, each tailored to accommodate ever wider classes of systems and each reflecting successively weaker shades of hyperbolic behavior.

Under the hypotheses of Theorem 1.1, our analysis uncovers a fine structure of the invariant measure space: a conclusion that substantially strengthens the intermediate entropy property. A subset $\Delta \subseteq X$ is termed $\mathcal{T}$-invariant if $\mathcal{T}(\Delta)\subset \Delta$. Whenever $\Delta$ is compact and $\mathcal{T}$-invariant, the pair $(\Delta,\mathcal{T})$ itself forms a topological dynamical system; consequently, notations such as $\mathcal{M}(\Delta,\mathcal{T})$ make sense. Motivated by the concepts of \emph{almost entropy-approximable}, \emph{entropy-approximable} and \emph{entropy-generic} introduced in \cite[Definition 1.2]{S25}, we adapt these notions to the setting of a continuous $L$-action $\mathcal{T}$.
\begin{definition}\label{definition 1.2} 
\cite[Definition 1.2]{S25} Let $(X,d,\mathcal{T})$ be a topological dynamical system.
\begin{enumerate}
    \item Given $\mu \in \mathcal{M}(X, \mathcal{T})$, we say that $\mu$ is \emph{almost entropy-approximable} (by compact invariant sets with intermediate entropies) if for every neighborhood $U$ of $\mu$, every $h\in (0, h_\mu(\mathcal{T}))$ and every $\varepsilon, \beta > 0$, there exist a compact $\mathcal{T}$-invariant set $\Delta$ and a $\gamma \in (0, \varepsilon)$ such that
    $$\mathcal{M}(\Delta, \mathcal{T}) \subset U,~~h(\Delta, \mathcal{T}) > h~~\text{and}~~h(\Delta, \mathcal{T}, \gamma) < h + \beta,$$ 
    where $h(\Delta, \mathcal{T}, \gamma)$ denotes the Bowen topological entropy of the subsystem $(\Delta, \mathcal{T})$ computed at the scale $\gamma$.
    \item Given $\mu \in \mathcal{M}(X, \mathcal{T})$, we say that $\mu$ is \emph{entropy-approximable} (by compact invariant sets with intermediate entropies) if for every neighborhood $U$ of $\mu$, every $h \in (0, h_\mu(\mathcal{T}))$ and every $\beta > 0$, there exists a compact $\mathcal{T}$-invariant set $\Delta$ such that
    $$\mathcal{M}(\Delta, \mathcal{T}) \subset U~~\text{and}~~h < h(\Delta, \mathcal{T}) < h + \beta.$$
    \item The system $(X, d, \mathcal{T})$ is said to be \emph{entropy-generic} if for every $\alpha \in [0, h(\mathcal{T}))$, the set
    $$\mathcal{M}_e(X, \mathcal{T}, \alpha):= \{ \mu \in \mathcal{M}_e(X, \mathcal{T}) : h_\mu(\mathcal{T}) = \alpha \}$$
    is residual in the subspace
    $$\mathcal{M}^\alpha(X, \mathcal{T}) := \{ \mu \in \mathcal{M}(X, \mathcal{T}) : h_\mu(\mathcal{T}) \geq \alpha \}.$$
\end{enumerate}
\end{definition}
The following theorem is our key result, which relies solely on the approximate $L$-product property.
\begin{theorem}\label{theorem 1.3}
Let $(X, d, \mathcal{T})$ be a system with the approximate $L$-product property. Then every invariant measure $\mu \in \mathcal{M}(X, \mathcal{T})$ is almost entropy-approximable.
\end{theorem}
If the system is further assumed to be asymptotically entropy expansive, a sharper conclusion can be drawn, which directly yields Theorem \ref{thm 1.1}. Indeed, entropy-genericity is a notion considerably stronger than the intermediate entropy property. Moreover, under these conditions, the existence of at least one measure of maximal entropy is ensured.
\begin{theorem}\label{theorem 1.4}
Let $(X, d, \mathcal{T})$ be an asymptotically entropy expansive system with the approximate $L$-product property. Then the following hold:
\begin{enumerate}
\item\label{theorem 1.4(1)} Every invariant measure $\mu \in \mathcal{M}(X, \mathcal{T})$ is entropy-approximable.
\item\label{theorem 1.4(2)} The system $(X, d, \mathcal{T})$ is entropy-generic.
\end{enumerate}
\end{theorem}
We remark that the definitions and conclusions discussed above remain meaningful in the zero-entropy setting, where the requirements concerning entropy estimates are automatically satisfied. For a measure $\mu$ with $h_\mu(\mathcal{T}) = 0$, it is (almost) entropy-approximable precisely when for every neighborhood $U$ of $\mu$ there exists a compact $\mathcal{T}$-invariant subset $\Delta$ such that $\mathcal{M}(\Delta, \mathcal{T}) \subset U$. In the case $h(\mathcal{T}) = 0$, entropy-genericity holds without further assumptions. Moreover, it follows from \cite[Theorem 2.1]{PS05} that for any system possessing the approximate $L$-product property, the collection $\mathcal{M}_e(X, \mathcal{T}) = \mathcal{M}_e(X, \mathcal{T}, 0)$ is residual in $\mathcal{M}(X, \mathcal{T}) = \mathcal{M}^0(X, \mathcal{T})$.

The notion of entropy-genericity was first introduced by Sun. Subsequently, Li and Oprocha \cite{LO18} were shown to have independently established entropy-genericity under more restrictive hypotheses, namely topological transitivity, the shadowing property and upper semi-continuity of the entropy map. Combining these findings with the description of the structure of $\mathcal{M}(X,\mathcal{T})$ provided in Theorem~\ref{theorem 1.4}, we obtain the following corollary.
\begin{corollary}\label{cor1.5}
Let $(X, d, \mathcal{T})$ be an asymptotically entropy expansive system with the approximate $L$-product property. For $\mu \in \mathcal{M}(X,\mathcal{T})$ and a neighborhood $U$ of $\mu$, denote
$$\mathbb{H}(X, \mathcal{T}, U) := \{h_\nu(\mathcal{T}) : \nu \in U \cap \mathcal{M}_e(X, \mathcal{T})\}.$$
Then for every $\mu \in \mathcal{M}(X,\mathcal{T})$ and every neighborhood $U$ of $\mu$, we have
\[\begin{cases} 
\mathbb{H}(X, \mathcal{T}, U) \supset [0, h_\mu(\mathcal{T})], & \text{if } h_\mu(\mathcal{T}) < h(\mathcal{T}); \\ 
\mathbb{H}(X, \mathcal{T}, U) \supset [0, h_\mu(\mathcal{T})), & \text{if } h_\mu(\mathcal{T}) = h(\mathcal{T}).
\end{cases}\]
\end{corollary}
Recall that the system $(X, d, \mathcal{T})$ is said to be \textbf{minimal} if it contains no nonempty proper compact $\mathcal{T}$-invariant subset. In Subsection \ref{subsection 4.5}, we shall prove the following result.  
\begin{corollary}\label{cor1.6}
 Let $(X, d, \mathcal{T})$ be a minimal system with the approximate $L$-product property. Then $(X, d, \mathcal{T})$ must be uniquely ergodic and $h(\mathcal{T}) = 0$.
\end{corollary}
\section{Preliminaries}
Throughout this paper, we write $\mathbb{R}$ for the set of real numbers, $\mathbb{Z}$ for the set of integers, and set $\mathbb{N} := \{1,2,\dots\}$, $\mathbb{Z}_+ := \mathbb{N} \cup \{0\}$.

Let $(X, d)$ be a compact metric space. By a continuous $L$-action we mean a family 
$\mathcal{T} := \{ T^{{\mathbf{i}}} : X \to X \}_{{\mathbf{i}}\in L}$ of continuous maps satisfying $T^{{\mathbf{h}}+{\mathbf{k}}} = T^{{\mathbf{h}}} \circ T^{{\mathbf{k}}}$~~for all $\mathbf{h}, \mathbf{k} \in L$\quad and $T^0 =\operatorname{id}_X.$ 
For a subset $\Lambda \subset L$ and $\mathbf{k}\in L$, we write $\mathbf{k} + \Lambda := \{ \mathbf{j} \in L: \mathbf{j} = \mathbf{k} + \mathbf{l}, \mathbf{l} \in \Lambda \}$. The cardinality of a finite set $B$ is denoted by $|B|$.

A particularly important class of examples is given by \emph{shift spaces}. Here $X$ is a nonempty closed $\mathcal{T}$-invariant subset of $A^L$, where $A = \{0,\dots,b-1\}$ is a finite alphabet with $b > 1$ symbols. The alphabet $A$ is equipped with the discrete topology, and $A^L$ with the corresponding product topology. Any metric compatible with this topology may be employed. The action $\mathcal{T}$ is induced by translation in the index set $L$: for each $\mathbf{j} \in L$ and every $\omega \in X$, we define $(T^{\mathbf{j}} \omega)_{\mathbf{k}} := \omega_{\mathbf{j}+\mathbf{k}}$ for all $\mathbf{k} \in L$.

Let $C(X)$ denote the space of continuous real‑valued functions on $X$. The action $\mathcal{T}$ on $X$ naturally induces an action on $C(X)$ by $(T^{\mathbf{i}} f)(x) := f(T^{\mathbf{i}} x)$. For a Borel measurable function $f : X \to \mathbb{R}$ and a measure $\nu \in \mathcal{M}(X)$, we use the notation $\langle f,\nu\rangle := \int f \, d\nu$. The action of $\mathcal{T}$ on $\mathcal{M}(X)$ is then defined by $\langle f, T^{\mathbf{i}} \nu \rangle = \langle T^{\mathbf{i}} f, \nu \rangle$ for every $f \in C(X)$ and every $\mathbf{i} \in L$.
\begin{definition}\label{definition 2.1}
 Let $\Lambda_n$ denote $[-n, n]^d$ when $L = \mathbb{Z}^d$, and $\Lambda_n := [0, n]^d$ when $L = \mathbb{Z}_+^d$.
Let $D_n$ denote the width of $\Lambda_n$: $D_n = 2n + 1$ if $L = \mathbb{Z}^d$, and $D_n = n + 1$ if $L = \mathbb{Z}_+^d$. Define $V_n := |\Lambda_n| = D_n^d$. For $m, n \in \mathbb{N}$, we define 
\begin{align}\label{2.1}
 m * n := 
\begin{cases} 
mn + m + n & \text{if } L = \mathbb{Z}_+^d; \\
2mn + m + n & \text{if } L = \mathbb{Z}^d.
\end{cases}   
\end{align}
Observe that $\Lambda_{m * n}$ can be expressed as a disjoint union of $V_{m * n}/V_m = V_n$ translates of $\Lambda_m$.
\end{definition}
\subsection{Topological entropy and metric entropy}
In this subsection, we introduce the concepts of topological entropy, metric entropy and the variational principle for topological dynamical systems.
\begin{definition}\label{definition2.2}
For any nonempty finite set $\Lambda \subset L$, we introduce the metric $d_{\Lambda}$ on $X$
by 
$$d_{\Lambda}(x, y) := \max \{ d(T^{\mathbf{i}} x, T^{\mathbf{i}} y): \mathbf{i}\in \Lambda \}.$$
If $\Lambda = \emptyset$, we simply define $d_{\Lambda} \equiv 0$.
A subset $F \subset X$ is said to be $(\Lambda, \varepsilon)$-separated if for any distinct points $x, y \in F$, we have $d_{\Lambda}(x, y) > \varepsilon$. When $\Lambda = \Lambda_n$, we write $(n, \varepsilon)$-separated for brevity.
\end{definition}
Let $s(F, n, \varepsilon)$ denote the maximum cardinality of an $(n, \varepsilon)$-separated subset of $F$. Define
$$h(F, \mathcal{T}, \varepsilon):=\; \limsup_{n \to \infty} \frac{\ln s(F, n, \varepsilon)}{V_n}.$$
The topological entropy of $\mathcal{T}$ on the set $F$ is then given by
$$h(F, \mathcal{T}):=\; \lim_{\varepsilon \to 0} h(F, \mathcal{T}, \varepsilon).$$
In particular, $h(\mathcal{T}):= h(X, \mathcal{T})$ is referred to as the topological entropy of the system $(X, d, \mathcal{T})$.
For each $ n \in \mathbb{N}$, $d_{\Lambda_n}$ defines a metric on $X$. Observe that $h(F, \mathcal{T}, \varepsilon)$ is non-decreasing as $\varepsilon\to 0$. Hence, we have
\begin{align}\label{2.2}
h(F, \mathcal{T}) = \sup \{ h(F, \mathcal{T}, \varepsilon) : \varepsilon > 0 \}. 
\end{align}
\begin{definition}\label{definition 2.3}
Given $\varepsilon > 0$, a set of the form
$$B_n(x, \varepsilon) = \{y \in X : d_{\Lambda_{n}}(x, y)< \varepsilon\}$$
is referred to as an \emph{$(n,\varepsilon)$-ball} of the system $(X, \mathcal{T})$.
\end{definition} 
\begin{definition}\cite[Theorem 1.1]{K80}\label{definition 2.4}
Let $\mu\in\mathcal{M}(X, \mathcal{T})$ and fix $\delta \in (0,1)$. Denote
$$r_\mu(n,\varepsilon,\delta) := \min \left\{ |\mathcal{U}| : 
\mathcal{U} \text{ is a collection of } (n,\varepsilon)\text{-balls such that } 
\mu\!\left(\bigcup_{U \in \mathcal{U}} U\right) > 1 - \delta \right\},$$
where $|\mathcal{U}|$ denotes the cardinality of $\mathcal{U}$. The \emph{metric entropy} of $(X, d, \mathcal{T})$ with respect to $\mu$ is then defined as 
$$h_\mu(\mathcal{T}) := \lim_{\varepsilon \to 0} \limsup_{n \to \infty} \frac{\ln r_\mu(n,\varepsilon,\delta)}{V_n} = \lim_{\varepsilon \to 0} \liminf_{n \to \infty} \frac{\ln r_\mu(n,\varepsilon,\delta)}{V_n}.$$
\end{definition}
The entropy map $\mu \mapsto h_\mu(\mathcal{T})$ is thus defined on $\mathcal{M}(X, \mathcal{T})$.
\begin{proposition}\cite[Theorem 8.1]{W82}\label{proposition 2.5}
For any $\mu,\nu\in\mathcal{M}( X,\mathcal{T})$ and $\lambda\in[0,1]$, we have
$$h_{\lambda\mu+(1-\lambda)\nu}(\mathcal{T})=\lambda h_{\mu}(\mathcal{T})+(1-\lambda)h_{\nu}(\mathcal{T}).$$
\end{proposition}
\begin{proposition}\cite[Theorem 8.6]{W82}\label{proposition 2.6}
Let $(X, d, \mathcal{T})$ be a topological dynamical system. Then
$$h(\mathcal{T}) = \sup \{ h_\mu(\mathcal{T}) : \mu \in \mathcal{M}(X, \mathcal{T}) \} = \sup \{ h_\mu(\mathcal{T}) : \mu \in \mathcal{M}_e(X, \mathcal{T})\}.$$
\end{proposition}
\subsection{Topological expansiveness}
In \cite{RS16}, Ren and Sun defined the notion of asymptotic entropy expansiveness for amenable group actions and proved that the entropy function is upper semi-continuous on the space of invariant Borel probability measures for asymptotically entropy expansive amenable group actions. In this subsection, we introduce some notation and results on entropy expansiveness for a continuous $L$-action $\mathcal{T}$.
\begin{definition}\label{definition 2.7}
A subset $E\subset X$ is said to be an \emph{$(n,\varepsilon)$-spanning set} if it satisfies
$$X = \bigcup_{x \in E} B_n(x, \varepsilon).$$ 
We denote by $r(n, \varepsilon)$ the minimum cardinality of an $(n, \varepsilon)$-spanning subset of $X$. 
\end{definition} 
\begin{definition}\cite[Definition 2.3]{S25} and \cite[Definition 3.1]{RS16}\label{definition 2.8}
For $\varepsilon > 0$ and $x \in X $, denote
$$\Gamma_\varepsilon(x) := \left\{ y \in X : d(T^{\mathbf{i}}(x), T^{\mathbf{i}}(y)) < \varepsilon \text{ for all } \mathbf{i} \in L\right\}.$$
Let
$$h^*(\mathcal{T}, \varepsilon) := \sup\{h(\Gamma_\varepsilon(x), \mathcal{T}):x \in X\}.$$
\begin{enumerate}
\item The system $(X, d, \mathcal{T})$ is \emph{expansive} if there exists $\varepsilon_0 > 0$ such that $\Gamma_{\varepsilon_0}(x) = \{x\}$ for every $ x\in X$.
\item The system $(X, d, \mathcal{T})$ is \emph{entropy expansive} if there exists $\varepsilon_0 > 0$ such that $h^*(\mathcal{T}, \varepsilon_0) = 0$.
\item The system $(X, d, \mathcal{T})$ is \emph{asymptotically entropy expansive} if
$$\lim_{\varepsilon \to 0} h^*(\mathcal{T}, \varepsilon) = 0.$$ 
\end{enumerate}
\end{definition}
\begin{proposition}\cite[Theorem 2]{RS16}\label{proposition 2.9}
For any subset $F \subset X$ and any $\varepsilon > 0$, we have
$$h(F, \mathcal{T}) \leq h(F, \mathcal{T}, \varepsilon) + h^*(\mathcal{T}, \varepsilon).$$
\end{proposition}
\begin{proposition}\cite[Theorem 1]{RS16}\label{proposition 2.10}
Let $(X, d, \mathcal{T})$ be asymptotically entropy expansive. Then the entropy map $\mu \mapsto h_\mu(\mathcal{T})$ is upper semi-continuous with respect to the weak-$*$ topology on $\mathcal{M}(X, \mathcal{T})$. Consequently, there exists $\tilde{\mu} \in \mathcal{M}_e(X, \mathcal{T})$, called a measure of maximal entropy, such that $h_{\tilde{\mu}}(\mathcal{T}) = h(\mathcal{T})$.
\end{proposition}
\subsection{Invariant measures}
In this subsection, we present some background on invariant measures. The compactness of $X$ implies that $\mathcal{M}(X)$ and $\mathcal{M}(X,\mathcal{T})$ are compact and metrizable spaces under the weak-$*$ topology \cite[Theorem 6.5 and 6.10]{W82}. By \cite[Theorem 6.4]{W82}, there exists a countable and separating set of continuous functions $\left\lbrace f_1, f_2, \ldots\right\rbrace $ with $0 \leq f_k \leq 1$ such that
$$D(\mu,\nu):= \| \mu - \nu \|:=\sum_{i=1}^\infty\frac{|\langle f_i, \mu-\nu\rangle|}{2^i}$$
defines a compatible metric for the weak-$*$ topology on $\mathcal{M}\left( X\right)$ \cite[Theorem 6.4]{W82} and satisfies $D(\mu,\nu) \le 1$ for any $\mu,\nu \in \mathcal{M}(X)$. Throughout the remainder of this paper, we endow $X$ with the metric $d$ defined by 
$$d(x,y):=D(\delta_x,\delta_y),$$
where $\delta_{x}$ denotes the Dirac measure concentrated at $x$.
\begin{proposition}\cite[Theorem 6.4]{W82}\label{proposition2.11}
The metric $D$ defined above satisfies
$$D(\sum_{k=1}^{n}a_k\mu_k,\sum_{k=1}^{n}a_k\nu_k)\leq\sum_{k=1}^{n}a_kD(\mu_k,\nu_k)$$	
for any $n\in \mathbb{N}$, any $\mu_1, \dots, \mu_n, \nu_1, \dots, \nu_n \in \mathcal{M}(X)$ and any $a_1, \dots, a_n > 0$ with $\sum_{k=1}^{n}a_k=1$.
\end{proposition}
Let $\operatorname{ext}(K)$ be the set of extreme points of a convex set $K$. According to \cite[Theorem 6.10]{W82}, we have $\mathcal{M}_{e}(X,\mathcal{T}) = \operatorname{ext}(\mathcal{M}(X,\mathcal{T}))$ and $\mathcal{M}(X,\mathcal{T})$ forms a Choquet simplex, meaning that every $\mu \in \mathcal{M}(X,\mathcal{T})$ can be uniquely expressed as the barycenter of a probability measure concentrated on $\operatorname{ext}(\mathcal{M}(X,\mathcal{T}))$. Furthermore, $\mathcal{M}_{e}(X,\mathcal{T})$ is a $G_{\delta}$ subset of $\mathcal{M}(X,\mathcal{T})$. If $\mathcal{M}_{e}(X,\mathcal{T})$ is dense in $\mathcal{M}(X,\mathcal{T})$, then it is automatically a residual subset, and in this case $\mathcal{M}(X,\mathcal{T})$ is a Poulsen simplex precisely when it contains more than one point. The foundational properties of the Poulsen simplex were investigated in detail in \cite{LOS78}. A selection of key facts is summarized below. For a thorough introduction to Choquet simplices, we direct the reader to \cite{P01}.
\begin{proposition}\cite{LOS78}\label{proposition 2.12}
\begin{enumerate}
\item A metrizable Choquet simplex $S$ is a Poulsen simplex if and only if $S$ is not a singleton and $\operatorname{ext}(S)$ is dense in $S$.
\item The Poulsen simplex is unique up to affine homeomorphism.
\item If $S$ is a Poulsen simplex, then $\operatorname{ext}(S)$ is homeomorphic to the Hilbert space $\ell^2$. Notably, $\operatorname{ext}(S)$ is arcwise connected by simple arcs.
\end{enumerate}
\end{proposition}
\subsection{The approximate product property}
In this subsection, we introduce the approximate $L$-product property.
\begin{definition}\cite[Definition 2.9 ]{PS05}\label{definition 2.13}
The dynamical system $(X, d, \mathcal{T})$ has the approximate $L$-product property if the following condition holds. Given any $\varepsilon > 0$ and $\delta > 0$, there exists $N(\varepsilon , \delta)\in \mathbb{N}$ such that for any $n\geq N(\varepsilon , \delta)$ and any family $\{x^{\mathbf{i}}\}_{\mathbf{i}\in L} \subset X$, there exist a family  $\{\Lambda^{\mathbf{i}}\}_{\mathbf{i}\in L}$, a family $\{\xi_{\mathbf{i} }\}_{\mathbf{i} \in L}$ and $x\in X$ satisfying the following conditions:
\begin{enumerate}
    \item For every $\mathbf{i}\in L$, $\Lambda^{\mathbf{i}} \subset \Lambda_n$, $\Lambda^{\mathbf{i}} + \xi_{\mathbf{i} } \subset \Lambda_n$ and $|\Lambda_n \setminus \Lambda^{\mathbf{i}}| \leq \delta V_n$;
    \item For every $\mathbf{i}\in L$, $d_{\Lambda^{\mathbf{i}}}(T^{\mathbf{i}D_n +\xi_{\mathbf{i}}}x, x^{\mathbf{i}}) \leq \varepsilon $.
\end{enumerate}
\end{definition}
\begin{definition}\cite[Definition 2.7]{PS05}\label{definition 2.14}
A measure $\mu \in \mathcal{M}(X, \mathcal{T})$ is said to be \emph{entropy-approachable} by ergodic measures if for any $\eta>0$ and any $h < h_{\mu}(\mathcal{T})$, there is $\nu\in\mathcal{M}_{e}(X, \mathcal{T})$ such that
$$D(\mu,\nu)<\eta~~\text{and}~~h_{\nu}(\mathcal{T}) > h.$$ 
The dynamical system $(X, d, \mathcal{T})$ is called \emph{entropy-dense} if every $\mu \in \mathcal{M}(X, \mathcal{T})$ is entropy-approachable by ergodic measures.
\end{definition}
In \cite{PS05}, Pfister and Sullivan proved that the approximate $L$-product property implies entropy-dense.
\begin{proposition}\cite[Theorem 2.1 ]{PS05}\label{proposition2.15}
 If $(X, d, \mathcal{T})$ has the approximate $L$-product property, then $(X, d,  \mathcal{T})$ is entropy-dense.
\end{proposition}
\section{Empirical measures}
This section collects some facts about empirical measures in the context of $L$, which will be used in the proof of the main results. The proofs of Lemma \ref{lemma3.2} and Lemma \ref{lemma3.3} follow the ideas of \cite[Section 5.3]{CLT20}, adapted to the framework introduced in Definitions \ref{definition 2.1} and \ref{definition 3.1}.
\begin{definition}\label{definition 3.1}
 For each $x \in X$ and a finite subset $\Lambda \subset L$, the empirical measure $\mathcal{E}_{\Lambda}(x)$ is defined as
$$\mathcal{E}_{\Lambda}(x) := \frac{1}{|\Lambda|} \sum_{\mathbf{i} \in \Lambda} \delta_{T^{\mathbf{i}} x},$$
where $\delta_{x}$ denotes the Dirac measure concentrated at $x$. In particular, for $\Lambda = \Lambda_n$ we define 
$$\mathcal{E}_{n}(x):=\mathcal{E}_{\Lambda_{n}}(x)= \frac{1}{V_{n}} \sum_{\mathbf{i} \in \Lambda_{n}} \delta_{T^{\mathbf{i}}x}.$$
\end{definition}
Given a set $U \subset \mathcal{M}(X, \mathcal{T})$, we write
$$X_{n, U} := \{ x \in X : \mathcal{E}_{n}(x) \in U \}.$$
A point $x \in X$ is called a \emph{generic point} for $\nu \in \mathcal{M}(X, \mathcal{T})$ if for every $f \in C(X)$,
$$\lim_{n} \langle f, \mathcal{E}_{n}(x) \rangle = \langle f, \nu \rangle.$$
Let $\mu \in \mathcal{M}(X,\mathcal{T})$ be a fixed invariant probability measure and $\eta > 0$. Denote
$$\mathcal{B}_{\eta} = \mathcal{B}_{\eta}(\mu) := \overline{B(\mu, \eta)} = \{ \nu \in \mathcal{M}(X, \mathcal{T}) : D(\mu, \nu) \le \eta \}.$$
For $N \in \mathbb{N}$, we define
$$Z_{N,\eta} = Z_{N,\eta}(\mu) := \{ x \in X : T^{\mathbf{k}}(x) \in X_{N, \mathcal{B}_{\eta}} \text{ for } \mathbf{k} \in L\}.$$
That is, $Z_{N,\eta} = \bigcap_{\mathbf{k} \in L} T^{-\mathbf{k}}(X_{N, \mathcal{B}_{\eta}})$.
Clearly $T^{\mathbf{k}}(Z_{N,\eta}) \subset Z_{N,\eta}$ for every $\mathbf{k} \in L$. According to \cite[Section 6.1]{W82}, the map $x \mapsto \mathcal{E}_{N}(x)$ is continuous and hence uniformly continuous due to the compactness of $X$. Consequently, each $X_{N, \mathcal{B}_{\eta}}$ is closed and  $Z_{N,\eta}$ is also compact.\\
For $\varepsilon > 0$, define
$$\operatorname{Var}(\varepsilon) := \max \{ D(\mathcal{E}_{0}(x),\, \mathcal{E}_{0}(y)) : d(x, y) \le \varepsilon,\; x, y \in X \}.$$
The uniform continuity of the map $x \mapsto \mathcal{E}_{0}(x) = \delta_{x}$ implies
\begin{align}\label{3.1}
\lim_{\varepsilon \to 0} \operatorname{Var}(\varepsilon) = 0.
\end{align}
Denote by $D^{*}$ the diameter of $\mathcal{M}(X)$, i.e.
$$D^{*} := \max \{ D(\mu, \nu) : \mu, \nu \in \mathcal{M}(X) \}.$$
\begin{lemma}\label{lemma3.2}
Let $N \in \mathbb{N}$ and $\nu \in \mathcal{M}(Z_{N,\eta}, \mathcal{T})$. Then  
$D(\nu, \mu) \leq \eta.$
\end{lemma}	
\begin{proof}
Assume that $\nu\in \mathcal{M}(Z_{N,\eta}, \mathcal{T})$ is ergodic. Then there exists a generic point $x \in Z_{N,\eta}$ for $\nu$ such that $\mathcal{E}_{n}(x)$ converge to $\nu$ as $n \to \infty$. We present the proof for $L = \mathbb{Z}_+^d$, the case $L = \mathbb{Z}^d$ is analogous with the natural symmetric modification of the index set $J$.
For each $n \in \mathbb{N}$, define 
$$ q_n := \left\lfloor\frac{D_n}{D_N} \right\rfloor-1.$$
Note that $q_n$ is the largest integer such that $D_n-D_N<(q_n+1)D_N \le D_n$.
Let
$$J:= \{\mathbf{j} \in L: 0 \leq j_i \leq q_n,\ i=1,\dots,d \}.$$
For each $\mathbf{j} \in J$, the translated cube
$$\Lambda_N^{(\mathbf{j})} := \mathbf{j} D_N + \Lambda_N = \{ \mathbf{j} D_N + \mathbf{i} : \mathbf{i} \in \Lambda_N \}$$
is entirely contained in $\Lambda_n$, since $j_i D_N + N \le q_n D_N + N \le n$. These cubes are pairwise disjoint. Their union is
$$ U_n := \bigcup_{\mathbf{j} \in J} \bigl( \mathbf{j} D_{N} + \Lambda_N \bigr)=\Lambda_{q_nD_{N}+N},$$
and we have
$|U_n| = V_{q_n D_N + N}.$
The remainder set $R_n := \Lambda_n \setminus U_n$ satisfies
$$|R_n| = |\Lambda_n| - |U_n| \le D_n^{d} - (D_n- D_N)^d\le d D_n^{d-1} D_N.$$
Thus,
\begin{align}\label{3.2}
 \frac{|R_n|}{|\Lambda_n|} \le \frac{dD_N}{D_n} \xrightarrow{n\to\infty} 0.
\end{align}
Define the empirical measure $\mathcal{E}_n(x)$ as a convex combination:
$$\mathcal{E}_{n}(x) = \frac{|U_n|}{|\Lambda_n|} \, \nu_n^{(U)} + \frac{|R_n|}{|\Lambda_n|} \, \nu_n^{(R)},$$
where   
$$\nu_n^{(U)} :=\frac{1}{|J|} \sum_{\mathbf{j}\in J} \mathcal{E}_{\Lambda_{N}}(T^{\mathbf{j} D_N} x)~~\text{and}~~\nu_n^{(R)} := \frac{1}{|R_n|} \sum_{\mathbf{i} \in R_n} \delta_{T^{\mathbf{i}} x}.$$
Since $x \in Z_{N,\eta}$, we have $D(\mathcal{E}_N(T^{\mathbf{j}} x), \mu) \leq \eta$  for every $\mathbf{j} \in L$. In particular, for each $\mathbf{j}$ with $0 \le j_i \le q_n$, $D(\mathcal{E}_N(T^{\mathbf{j} D_N} x), \mu) \le \eta$. By convexity of the metric $D$, 
 \begin{align}\label{3.3}
 D\bigl( \nu_n^{(U)}, \mu \bigr) \leq \eta.    
 \end{align}
Applying the convexity of $D$ again,  
$$D\bigl( \mathcal{E}_{n}(x), \mu \bigr) \leq \frac{|U_n|}{|\Lambda_n|} \eta + \frac{|R_n|}{|\Lambda_n|} D^*.$$
By (\ref{3.3}) and the continuity of $D$, we have
$$\lim_{n\to\infty} D\bigl( \mathcal{E}_{n}(x), \mu \bigr)=D(\nu, \mu) \leq \eta.$$ 
If $\nu$ is not ergodic, the same conclusion follows from the ergodic decomposition together with the convexity of $D$ (Proposition \ref{proposition2.11}).
\end{proof}
\begin{lemma}\label{lemma3.3}
Let $\eta, \delta, \varepsilon > 0$ and $K, M \in \mathbb{N}$ such that
\begin{align}\label{3.4}
 \frac{2d D^{*}}{K} < \eta~~\text{and}~~\operatorname{Var}(\varepsilon) + 2\delta D^{*} < \eta. 
\end{align}
Suppose that a sequence of points $\mathcal{C} = \{x_{\mathbf{v}}\}_{\mathbf{v}\in L} \subset X_{ M, B(\mu, \eta)}$ is traced by $z\in X$ in the sense of definition \ref{definition 2.13}. Then $z \in Z_{KM, 3\eta}$.
\end{lemma}
\begin{remark}
By the definition of $Z_{KM, 3\eta}$, we have $z \in Z_{KM, 3\eta}$ if and only if $D\bigl(\mathcal{E}_{\Lambda_{KM}}(T^{\mathbf{k}} z), \mu\bigr) < 3\eta$ for every $\mathbf{k} \in L$. Since the system is translation invariant, it suffices to prove $D\bigl(\mathcal{E}_{\Lambda_{KM}}(z), \mu\bigr) < 3\eta$. The case of arbitrary $\mathbf{k}$ follows by applying the same argument to $T^{\mathbf{k}} z$ and the translated sequence $\{x_{\mathbf{v}+\mathbf{k}}\}_{\mathbf{v}\in L}$. Hence in the proof below we take $\mathbf{k}=0$.
\end{remark}
\begin{proof}
We prove $D\bigl(\mathcal{E}_{\Lambda_{KM}}(z), \mu\bigr) < 3\eta$ in four steps. We present the proof for $L = \mathbb{Z}_+^d$, the case $L = \mathbb{Z}^d$ is analogous with the natural symmetric modification of the index set $I$.\\
\textbf{Step 1: Decomposition of the domain $\Lambda_{KM}$.}\\
Define
$$q := \Bigl\lfloor \frac{D_{KM}}{D_{M}} \Bigr\rfloor - 1,$$
so that $q$ is the largest integer satisfying $D_{KM}-D_M<(q+1)D_M \le D_{KM}$. Let
$$I:= \{\mathbf{v} \in L: 0 \leq v_i \leq q,\ i=1,\dots,d \}.$$
For each $\mathbf{v}\in I$, the translated cube
$$\Lambda_M^{(\mathbf{v})} := \mathbf{v} D_M + \Lambda_M = \{\mathbf{v} D_M + \mathbf{i} : \mathbf{i} \in \Lambda_M \}$$
is entirely contained in $\Lambda_{KM}$, since $v_i D_M + M \le q D_M + M\le KM$. These cubes are pairwise disjoint. Their union is
$$ U:= \bigcup_{\mathbf{v}\in I} \bigl(\mathbf{v} D_{M} + \Lambda_M \bigr)=\Lambda_{qD_{M}+M},$$
and we have $|U|=V_{qD_{M}+M}$.\\
\textbf{Step 2: Estimate of the remainder.}\\
The uncovered part $\Lambda_{KM}\setminus U$ satisfies
$$|\Lambda_{KM} \setminus U|= |\Lambda_{KM}| - |U|\le D_{KM}^d - (D_{KM}-D_M)^d\le d D_{M}(D_{KM})^{d-1}.$$
Consequently,
$$\frac{| \Lambda_{KM}\setminus U|}{|\Lambda_{KM}|} \leq \frac{d D_M}{D_{KM}} < \frac{dD_{M}}{D_{KM}-1}.$$
Since $M \geq 1$, we have $\frac{D_{M}}{D_{KM}-1}<\frac{2}{K}$ and therefore
\begin{align}\label{3.5}
\frac{|\Lambda_{KM} \setminus U|}{|\Lambda_{KM}|} < \frac{2d}{K}.    
\end{align}
\textbf{Step 3: Splitting the empirical measure.}\\
Decompose $\mathcal{E}_{\Lambda_{KM}}(z)$ as
$$\mathcal{E}_{\Lambda_{KM}}(z) = \frac{|U|}{|\Lambda_{KM}|} \mathcal{E}_{U}(z) + \frac{|\Lambda_{KM}\setminus U|}{|\Lambda_{KM}|} \mathcal{E}_{\Lambda_{KM} \setminus U}(z).$$
Using the convexity of $D$ together with \eqref{3.4} and \eqref{3.5}, we have
\begin{align}\label{3.6}
 D\bigl(\mathcal{E}_{\Lambda_{KM}}(z), \mu\bigr) \leq \frac{|U|}{|\Lambda_{KM}|} D\bigl(\mathcal{E}_{U}(z), \mu\bigr) + \eta.    
\end{align}
\textbf{Step 4: Estimating the empirical measure on $U$.}\\
Since $U$ is the disjoint union of the cubes $\Lambda_M^{(\mathbf{v})}$, we have
$$\mathcal{E}_{U}(z) = \frac{1}{|I|} \sum_{\mathbf{v} \in I} \mathcal{E}_{\Lambda_M}(T^{\mathbf{v} D_{M}} z)=: \frac{1}{|I|} \sum_{\mathbf{v}\in I} \nu_{\mathbf{v}} .$$
Fix $\mathbf{v} \in I$. By definition \ref{definition 2.13}, there exist a subset $\Lambda_{\mathbf{v}} \subset \Lambda_M$ and an index $\xi_{\mathbf{v}} \in L$ such that
\begin{enumerate}
\item $\Lambda_{\mathbf{v}} + \xi_{\mathbf{v}} \subset \Lambda_M$, $|\Lambda_M \setminus \Lambda_{\mathbf{v}}| \leq \delta V_M$;
\item $d_{\Lambda_{\mathbf{v}}}(T^{\mathbf{v} D_{M}+ \xi_{\mathbf{v}}} z, x_{\mathbf{v}}) \leq \varepsilon$.
\end{enumerate}
Define the following auxiliary measures:
$$x_{\mathbf{v}}^{\mathrm{target}} := \frac{1}{|\Lambda_{\mathbf{v}}|} \sum_{\mathbf{i} \in \Lambda_{\mathbf{v}}} \delta_{T^{\mathbf{i}} x_{\mathbf{v}}}~~\text{and}~~\nu_{\mathbf{v}}^{\mathrm{good}} := \frac{1}{|\Lambda_{\mathbf{v}}|} \sum_{\mathbf{i} \in \Lambda_{\mathbf{v}}} \delta_{T^{\mathbf{v} D_{M} + \xi_{\mathbf{v}} + \mathbf{i}} z}.$$
Define $\mathcal{E}_M(x_{\mathbf{v}})$ and $\nu_{\mathbf{v}}$ as convex combinations:
$$\mathcal{E}_M(x_{\mathbf{v}}) = \alpha x_{\mathbf{v}}^{\mathrm{target}} + (1-\alpha)\tilde{x}_{\mathbf{v}}~~\text{and}~~\nu_{\mathbf{v}} = \alpha\nu_{\mathbf{v}}^{\mathrm{good}} + (1-\alpha)\nu_{\mathbf{v}}^{\mathrm{bad}}.$$
where $\alpha := \frac{|\Lambda_{\mathbf{v}}|}{V_M}\geq 1-\delta$, $\tilde{x}_{\mathbf{v}}$ is the empirical measure of $x_{\mathbf{v}}$ on $\Lambda_M \setminus \Lambda_{\mathbf{v}}$ and $\nu_{\mathbf{v}}^{\mathrm{bad}}$ is the empirical measure of $\nu_{\mathbf{v}}$ on $\Lambda_M \setminus (\Lambda_{\mathbf{v}}+\xi_{\mathbf{v}})$. From condition (2) above we have
\begin{align}\label{3.7}
 D\bigl(\nu_{\mathbf{v}}^{\mathrm{good}}, x_{\mathbf{v}}^{\mathrm{target}}\bigr) \leq  \text{Var}(\varepsilon).   
\end{align}
Since $x_{\mathbf{v}} \in X_{ M, B(\mu, \eta)}$, we have $D(\mathcal{E}_M(x_{\mathbf{v}}),\mu) \leq \eta$. By convexity of $D$,
\begin{align}\label{3.8}
 D\bigl(x_{\mathbf{v}}^{\mathrm{target}}, \mu\bigr) \leq D\bigl(x_{\mathbf{v}}^{\mathrm{target}}, \mathcal{E}_M(x_{\mathbf{v}})\bigr)+D\bigl(\mathcal{E}_M(x_{\mathbf{v}}), \mu\bigr)\leq\delta D^{*} + \eta.   
\end{align}
Combining \eqref{3.7} and \eqref{3.8} yields
\begin{align}\label{3.9}
 D\bigl(\nu_{\mathbf{v}}^{\mathrm{good}}, \mu\bigr) \leq \text{Var}(\varepsilon) + \delta D^{*} + \eta.   
\end{align}
 Since $\Lambda_{\mathbf{v}} + \xi_{\mathbf{v}} \subset \Lambda_M$ and $|\Lambda_M \setminus (\Lambda_{\mathbf{v}} + \xi_{\mathbf{v}})| \leq \delta V_M$, we have
\begin{align}\label{3.10}
 D\bigl(\nu_{\mathbf{v}}, \nu_{\mathbf{v}}^{\mathrm{good}}\bigr) \leq \delta D^{*}.    
\end{align}
From \eqref{3.9} and \eqref{3.10},
\begin{align}\label{3.11}
 D(\nu_{\mathbf{v}}, \mu) \leq \text{Var}(\varepsilon) + 2\delta D^{*} + \eta.   
\end{align}
By \eqref{3.4}, we have $D(\nu_{\mathbf{v}}, \mu) < 2\eta$.
Since $\mathcal{E}_{U}(z)$ is a convex combination of the measures $\{\nu_{\mathbf{v}}\}_{{\mathbf{v}} \in I}$,  the convexity of $D$ implies
\begin{align}\label{3.12}
D\bigl(\mathcal{E}_{U}(z), \mu\bigr) < 2\eta.    
\end{align}
Applying \eqref{3.12} to \eqref{3.6},
$$D\bigl(\mathcal{E}_{\Lambda_{KM}}(z), \mu\bigr) \leq\frac{|U|}{|\Lambda_{KM}|} 2\eta + \eta < 2\eta + \eta = 3\eta.$$
Thus $D\bigl(\mathcal{E}_{\Lambda_{KM}}(z), \mu\bigr) < 3\eta$. For an arbitrary ${\mathbf{k}} \in L$, applying the same argument to $T^{{\mathbf{k}}} z$ and the shifted sequence $\{x_{{\mathbf{v}}+{\mathbf{k}}}\}_{{\mathbf{v}} \in L}$, we also have $D(\mathcal{E}_{\Lambda_{KM}}(T^{\mathbf{k}} z), \mu) < 3\eta$. Therefore $z \in Z_{KM, 3\eta}$.
\end{proof}
For the entropy estimate, we will also need the following results from Pfister and Sullivan \cite{PS05}.
\begin{definition}\label{definition 3.5}
Let $F \subset X$. For $n \in \mathbb{N}$, $\delta > 0$ and $\varepsilon > 0$, a subset $E \subset F$ is called $(n,\delta,\varepsilon)$-separated if for any distinct points $x, y \in E$, we have
$$\bigl|\{\mathbf{j} \in \Lambda_n : d(T^{\mathbf{j}} x, T^{\mathbf{j}} y) > \varepsilon \}\bigr| > \delta \, V_n,$$
or equivalently, if for every $\Lambda \subset \Lambda_n$ with $|\Lambda_n \setminus \Lambda| \leq \delta V_n$, we have $d_\Lambda(x, y) > \varepsilon.$
\end{definition}
\begin{proposition}\cite[Proposition 2.1]{PS05}\label{pro3.6}
Let $(X, d, \mathcal{T})$ be a topological dynamical system.  
Suppose that $\nu \in \mathcal{M}_e(X, \mathcal{T})$ and $h < h_\nu(\mathcal{T})$. Then there exist $\delta > 0$ and $\varepsilon > 0$ such that 
for any neighborhood $U$ of $\nu$, there exists $N^* = N^*(h, \delta,\varepsilon, U) > 0$ such that for every $n \geq N^*$, there is an $(n, \delta, \varepsilon)$-separated set $\Gamma_n \subset X_{n, U}$ satisfying $|\Gamma_n| \geq e^{V_nh}$.
\end{proposition}
\begin{lemma}\cite[Lemma 2.1]{PS05}\label{lem3.7}
For $ n \in\mathbb{N} $ and $ \delta \in (0, \frac{1}{2})$,	denote  
$$ Q(n, \delta) := |\{A \subset \Lambda _n : |A| \geq (1 - \delta)V_n\}|.$$
Then  
\begin{align}
 \frac{\ln Q(n, \delta)}{V_n} \leq -\delta \ln \delta - (1 - \delta) \ln(1 - \delta).   
\end{align}
Moreover,
\begin{align}
 \lim_{\delta \to 0} \bigl(-\delta \ln \delta - (1 - \delta) \ln(1 - \delta)\bigr) = 0.   
\end{align}
\end{lemma}
\section{ Almost entropy-approximability}
 In this section, we prove Theorem \ref{theorem 1.3}, the centrepiece of the paper. Motivated by insights from Sun \cite{S25}, we construct new compact invariant sets and obtain fine estimates of their entropies. We establish Theorem \ref{theorem 1.3} for an arbitrary invariant measure $\mu$ by approximating it with convex combinations of ergodic measures. Proposition \ref{proposition2.15} is employed to streamline the argument, reducing the task to proving that every ergodic measure is entropy-approximable.                         
\begin{proposition}	\label{proposition 4.1}
 Let $(X, d, \mathcal{T})$ be a system with the approximate $L$-product property. Suppose that $\mu_0 \in \mathcal{M}_e(X, \mathcal{T})$, $h_0 \in (0, h_{\mu_0}(\mathcal{T}))$ and $\eta_0, \beta_0, \varepsilon_0 > 0$. Then there exist $\gamma\in(0, \varepsilon_0)$ and a compact $\mathcal{T}$-invariant subset $ \Delta =\Delta(\mu_0,h_0,\eta_0,\beta_0,\gamma)$ such that:
\begin{enumerate}
\item\label{pro 4.1(1)} For every $\nu \in \mathcal{M}(\Delta, \mathcal{T})$, we have $D(\nu, \mu_0) < \eta_0$.
\item $h(\Delta , \mathcal{T}) > h_0$ and $h(\Delta , \mathcal{T}, \gamma) < h_0 + \beta_0$.
\end{enumerate}	
\end{proposition}
We prove Proposition \ref{proposition 4.1} in subsections \ref{con4.1}, \ref{entropy lower 4.2} and \ref{upper bound 4.3}.
\subsection{Construction}\label{con4.1}
In this subsection, we present the proof of Proposition \ref{proposition 4.1} (\ref{pro 4.1(1)}).\\
Assume that $(X, d, \mathcal{T})$ has the approximate $L$-product property. Let $\mu_0 \in  \mathcal{M}_e(X, \mathcal{T})$, $h_0 \in (0, h_{\mu_0}(\mathcal{T}))$ and $\eta_0, \beta_0, \varepsilon_0 > 0$ be given. We fix
$$\eta := \frac{\eta_0}{4}, \quad 
\beta := \frac{1}{20} \min\left\{ \beta_0, h_{\mu_0}(\mathcal{T}) - h_0, h_0 \right\},
h_1 := h_0 + 10\beta$$
and choose $K\in\mathbb{N}$ such that 
\begin{align}\label{4.1}
 K\eta > 2d D^*.
\end{align}
Note that $h_1 + \beta < h_{\mu_0}(\mathcal{T})$. By Proposition \ref{pro3.6}, there exist $\delta_0 > 0$, $\gamma_0 > 0$ and $N^* = N^*(h_1 + \beta,\delta_0, \gamma_0,B(\mu_0, \eta) )$ such that for every $n \ge N^*$, there exists an $(n, \delta_0, \gamma_0)$-separated set $\Gamma_n^* \subset X_{n, B(\mu_0, \eta)}$ satisfying
    \begin{align}\label{4.2}
    |\Gamma_n^*| > e^{V_n (h_1 + \beta)}.    
    \end{align}
By \eqref{3.1}, we can select $\varepsilon > 0$ such that
\begin{align}\label{4.3}
\operatorname{Var}(\varepsilon) < \frac{\eta}{4}~~\text{and}~~\varepsilon < \frac{1}{3}\min\{\varepsilon_0, \gamma_0\}.     
\end{align}
Let $r(\varepsilon)$ denote the minimal cardinality of an $(1, \varepsilon)$-spanning set for $X$.\\
Choose $\delta \in (0, \frac{1}{2})$ satisfying
\begin{align}\label{4.4}
\delta < \min\left\{ \frac{\delta_0}{2}, \frac{1}{2K}, \frac{\beta}{\ln r(\varepsilon)} \right\}    
\end{align}
and 
\begin{align}\label{4.5}
0 < -\delta \ln \delta - (1-\delta)\ln(1-\delta) < \beta.    
\end{align}
Let $N(\varepsilon, \delta)$ for $(X, d, \mathcal{T})$ be as in definition \ref{definition 2.13}. Fix $M\in\mathbb{N}$ such that
\begin{align}\label{4.6}
   M > \max\{N(\varepsilon, \delta), N^*\},~~\frac{\ln V_M}{V_M} < \beta~~\text{and}~~e^{V_M (h_1 + \beta)} > e^{V_M h_1} +1.
\end{align}	
By Proposition \ref{pro3.6} and \eqref{4.2}, there exists an $(M, \delta_0, \gamma_0)$-separated set $\Gamma_M^* \subset X_{M, B(\mu_0, \eta)}$ with $|\Gamma_M^*| > e^{V_M (h_1 + \beta)}$. Choose a subset $\Gamma_M \subset \Gamma_M^*$ such that
\begin{align}\label{4.7}
e^{V_M h_1} \le |\Gamma_M| < e^{V_M (h_1 + \beta)}.   
\end{align}
Define the symbolic spaces: 
$$\Gamma := (\Gamma_M)^{L}~~\text{and}~~\Sigma := (\Lambda_M)^{L}.$$
For any $\mathscr{C} = \{x_{\mathbf{i}}(\mathscr{C})\}_{\mathbf{i} \in L} \in \Gamma $, since $M > N(\varepsilon, \delta)$, there exist families $\{\Lambda^{\mathbf{i}}\}_{\mathbf{i}\in L}$, $\xi=\{\xi_{\mathbf{i}}\}_{\mathbf{i}\in L}\in \Sigma$ and a point $y \in X$ satisfying:
\begin{enumerate}
 \item For every $\mathbf{i}\in L$, $\Lambda^{\mathbf{i}} \subset \Lambda_M$, $\Lambda^{\mathbf{i}} + \xi_{\mathbf{i} } \subset \Lambda_M$, and $|\Lambda_M \setminus \Lambda^{\mathbf{i}}| \leq \delta V_M$;
\item For every $\mathbf{k} \in \Lambda^{\mathbf{i}}$, $d(T^{\mathbf{i} D_M + \xi_{\mathbf{i}} + \mathbf{k}} y,\; T^{\mathbf{k}} x_{\mathbf{i}}) \le \varepsilon.$
\end{enumerate}
Denote by $Y_{\mathscr{C},\xi}$ the set of all such points $y$ and define
$$Y := \bigcup_{\mathscr{C}\in\Gamma, \xi\in\Sigma} Y_{\mathscr{C}, \xi}.$$
By \eqref{4.1}, \eqref{4.3} and \eqref{4.4}, we have
$$\operatorname{Var}(\varepsilon) +  2\delta D^* < \frac{1}{4}\eta + \frac{D^*}{K} < \eta.$$
Hence condition \eqref{3.3} holds. By Lemma \ref{lemma3.3}, we have $Y \subset Z_{KM,3\eta}$.

Let $\sigma_\Gamma$ and $\sigma_\Sigma$ denote the shift maps on $\Gamma$ and $\Sigma$, respectively.
\begin{lemma}\label{Lemma 4.2}
For each $\mathscr{C} \in \Gamma$, $\xi \in \Sigma$ and every standard basis vector $\mathbf{e}_j = (0,\dots,1,\dots,0)$ (with $1$ in the $j$-th coordinate and $0$ elsewhere), we have
$$T^{D_M \mathbf{e}_j}(Y_{\mathscr{C},\xi})\subset Y_{\sigma_{\mathbf{e}_j}(\mathscr{C}), \sigma_{\mathbf{e}_j}(\xi)},$$
where $\sigma_{\mathbf{e}_j}$ denotes the shift map on $L$, defined by $\sigma_{\mathbf{e}_j}(\mathscr{C}) = \{x_{\mathbf{i}+\mathbf{e}_j}\}_{\mathbf{i} \in L}$, $\sigma_{\mathbf{e}_j}(\xi) = \{\xi_{\mathbf{i}+\mathbf{e}_j}\}_{\mathbf{i} \in L}$.
\end{lemma}
\begin{proof}
Let $y \in Y_{\mathscr{C},\xi}$, together with families $\{\Lambda^{\mathbf{i}}\}_{\mathbf{i}\in L}$ and $\{\xi_{\mathbf{i}}\}_{\mathbf{i}\in L}$ satisfying:
\begin{enumerate}
\item For every $\mathbf{i}\in L$, $\Lambda^{\mathbf{i}} \subset \Lambda_M$, $\Lambda^{\mathbf{i}} + \xi_{\mathbf{i} } \subset \Lambda_M$, and $|\Lambda_M \setminus \Lambda^{\mathbf{i}}| \leq \delta V_M$;
\item For every $\mathbf{k} \in \Lambda^{\mathbf{i}}$, $d\bigl(T^{\mathbf{i} D_M + \xi_{\mathbf{i}} + \mathbf{k}} y,\; T^{\mathbf{k}} x_{\mathbf{i}}\bigr) \le \varepsilon$.
\end{enumerate}
Fix a standard basis vector $\mathbf{e}_j$ and let $y' = T^{D_M \mathbf{e}_j}(y)$. We aim to show that $y' \in Y_{\sigma_{\mathbf{e}_j}(\mathscr{C}), \sigma_{\mathbf{e}_j}(\xi)}$. Define a new family of subsets $\{\Lambda'^{\mathbf{i}}\}_{\mathbf{i} \in L}$ by $\Lambda'^{\mathbf{i}} := \Lambda^{\mathbf{i}+\mathbf{e}_j}$ and a new family of indices $\{\xi'_{\mathbf{i}}\}_{\mathbf{i} \in L}$ by $\xi'_{\mathbf{i}} := \xi_{\mathbf{i}+\mathbf{e}_j}$. Then for any $\mathbf{i} \in L$:
\begin{itemize}
\item $\Lambda'^{\mathbf{i}} = \Lambda^{\mathbf{i}+\mathbf{e}_j} \subset \Lambda_M$, $\Lambda'^{\mathbf{i}} + \xi'_{\mathbf{i}} = \Lambda^{\mathbf{i}+\mathbf{e}_j} + \xi_{\mathbf{i}+\mathbf{e}_j} \subset \Lambda_M$, and $|\Lambda_M \setminus \Lambda'^{\mathbf{i}}| = |\Lambda_M \setminus \Lambda^{\mathbf{i}+\mathbf{e}_j}| \le \delta V_M$.
\item For every $\mathbf{k} \in \Lambda'^{\mathbf{i}}$,
\[
\begin{aligned}
d\bigl(T^{\mathbf{i} D_M + \xi'_{\mathbf{i}} + \mathbf{k}} y',\; T^{\mathbf{k}} x_{\mathbf{i}+\mathbf{e}_j}\bigr)
&= d\bigl(T^{\mathbf{i} D_M + \xi_{\mathbf{i}+\mathbf{e}_j} + \mathbf{k}} (T^{D_M \mathbf{e}_j} y),\; T^{\mathbf{k}} x_{\mathbf{i}+\mathbf{e}_j}\bigr) \\
&= d\bigl(T^{(\mathbf{i}+\mathbf{e}_j) D_M + \xi_{\mathbf{i}+\mathbf{e}_j} + \mathbf{k}} y,\; T^{\mathbf{k}} x_{\mathbf{i}+\mathbf{e}_j}\bigr) \le \varepsilon.
\end{aligned}
\]
\end{itemize}
Thus $y' \in Y_{\sigma_{\mathbf{e}_j}(\mathscr{C}), \sigma_{\mathbf{e}_j}(\xi)}$.
\end{proof}
\begin{lemma}\label{lemma 4.3}
Let $\{y_n\}_{n=1}^\infty$ be a sequence in $Y$ such that $y_n \to \tilde{y}$ in $X$. Then there exist $\tilde{\mathscr{C}}\in \Gamma$ and $\tilde{\xi} \in \Sigma$ such that $\tilde{y} \in Y_{\tilde{\mathscr{C}},\tilde{\xi}}$. Consequently, $Y$ is closed and hence compact.
\end{lemma}
\begin{proof}
Denote
$$\Upsilon:= \{ A \subset \Lambda_M : |A| \ge (1-\delta) V_M \}.$$
By Lemma \ref{lem3.7}, $|\Upsilon|=Q(M,\delta)$. For each $\mathbf{i} \in L$ and each $y \in Y_{\mathscr{C},\xi}$ with associated family $\{\Lambda^{\mathbf{i}}\}$, we have $\Lambda^{\mathbf{i}} \in \Upsilon$.\\
Assume that $y_n \in Y_{\mathscr{C}_n,\xi_n}$ with a family $\{\Lambda^{\mathbf{i}}_n\}_{\mathbf{i}\in L}$. Since $\Gamma$, $\Sigma$ and $\Upsilon^L$ are compact metric symbolic spaces, we can find a subsequence $\{n_j\}$ and elements $\tilde{\mathscr{C}}\in \Gamma$, $\tilde{\xi}=\{\tilde{\xi}_{\mathbf{i}}\}_{\mathbf{i}\in L}\in \Sigma$ and $\{\tilde{\Lambda}^{\mathbf{i}}\}_{\mathbf{i}\in L}\in \Upsilon^L$ such that
$$\mathscr{C}_{n_j} \to \tilde{\mathscr{C}},~~\xi_{n_j} \to \tilde{\xi}~~\text{and}~~\{\Lambda^{\mathbf{i}}_{n_j}\}_{\mathbf{i}\in L} \to \{\tilde{\Lambda}^{\mathbf{i}}\}_{\mathbf{i}\in L}.$$
 Consequently, for each fixed $\mathbf{i}\in L$, there exists $J_{\mathbf{i}}\in\mathbb{N}$ such that for every $n_j \ge J_{\mathbf{i}}$, we have
$$x_{\mathbf{i}}(\mathscr{C}_{n_j}) = x_{\mathbf{i}}(\tilde{\mathscr{C}}),\quad \xi_{n_j,\mathbf{i}} = \tilde{\xi}_{\mathbf{i}}~~\text{and}~~\Lambda^{\mathbf{i}}_{n_j} = \tilde{\Lambda}^{\mathbf{i}}.$$
In particular, $\tilde{\Lambda}^{\mathbf{i}} + \tilde{\xi}_{\mathbf{i}} \subset \Lambda_M$ and $|\Lambda_M \setminus \tilde{\Lambda}^{\mathbf{i}}| \le \delta V_M$.\\
For each $\mathbf{i} \in L$ and $\mathbf{k}\in \tilde{\Lambda}^{\mathbf{i}}$, we have
$$d\bigl(T^{\mathbf{i}D_M + \tilde{\xi}_{\mathbf{i}} + \mathbf{k}} \tilde{y},\; T^{\mathbf{k}} x_{\mathbf{i}}(\tilde{\mathscr{C}})\bigr) = \lim_{j\to\infty} d\bigl(T^{\mathbf{i}D_M + \xi_{n_j,\mathbf{i}} + \mathbf{k}} y_{n_j},\; T^{\mathbf{k}} x_{\mathbf{i}}(\mathscr{C}_{n_j})\bigr) \le \varepsilon.$$
Thus $\tilde{y}\in Y_{\tilde{\mathscr{C}},\tilde{\xi}}$.
\end{proof}
\begin{proposition}
Let
\begin{align}\label{4.8}
\Delta := \bigcup_{\mathbf{k} \in \Lambda_{M}} T^{\mathbf{k}} Y.        
\end{align}
 Then $\Delta$ is a compact $\mathcal{T}$-invariant subset of $Z_{KM,3\eta}$. In particular, $\Delta$ satisfies Proposition \ref{proposition 4.1} (\ref{pro 4.1(1)}).
\end{proposition}
\begin{proof}
Since $Y$ is compact, $\Delta$ is compact as a finite union of translates of $Y$. Moreover, $Y \subset Z_{KM,3\eta}$ and $Z_{KM,3\eta}$ is $\mathcal{T}$-invariant, hence $\Delta \subset Z_{KM,3\eta}$.

To prove $\mathcal{T}$-invariance, it suffices to show that $T^{\mathbf{e}_j}(\Delta) \subset \Delta$ for each standard basis vector $\mathbf{e}_j = (0,\dots,1,\dots,0)$. Take any $z \in \Delta $. By \eqref{4.8}, there exist $y \in Y$ and $\mathbf{k} = (k_1,\dots,k_d) \in \Lambda_{M}$ such that $z = T^{\mathbf{k}} y$.\\	
When $\Lambda_{M}=[0,M]^d$, $D_{M}=M+1$;\\
\noindent\textbf{Case 1:}\label{case1} $k_j < D_M-1$. Then $\mathbf{k} + \mathbf{e}_j \in \Lambda_{M}$ and hence $T^{\mathbf{e}_j}(z) = T^{\mathbf{k}+\mathbf{e}_j} y \in \Delta$.\\	
\noindent\textbf{Case 2:} $k_j = D_M-1$. $\mathbf{k} = (k_1,\dots,k_{j-1}, M, k_{j+1},\dots,k_d)$. There are $\mathscr{C}$, $\xi$ such that $y \in Y_{\mathscr{C}, \xi}$. Consider $y' = T^{D_M \mathbf{e}_j} y$. By Lemma \ref{Lemma 4.2}, we have $y' \in Y$, thus
$$T^{\mathbf{e}_j}(z) = T^{\mathbf{k}+\mathbf{e}_j} y = T^{\mathbf{k}+\mathbf{e}_j- D_M\mathbf{e}_j} (T^{D_M \mathbf{e}_j} y) = T^{\mathbf{k} +\mathbf{e}_j- D_M\mathbf{e}_j} y' \in T^{\mathbf{k}+\mathbf{e}_j - D_M\mathbf{e}_j}(Y)\subset\Delta,$$
where $\mathbf{k}+\mathbf{e}_j- D_M\mathbf{e}_j=(k_1,\dots,k_{j-1}, 0, k_{j+1},\dots,k_d)\in \Lambda_M$. \\
When $\Lambda_{M}=[-M,M]^d$, $D_{M}=2M+1$;\\
\noindent\textbf{Case 3:}$k_j<\frac{D_M-1}{2}$. Then $\mathbf{k} + \mathbf{e}_j \in \Lambda_{M}$ and hence $T^{\mathbf{e}_j}(z) = T^{\mathbf{k}+\mathbf{e}_j} y \in \Delta$.\\
\noindent\textbf{Case 4:} $k_j=\frac{D_M-1}{2}$. Then, by the same argument as in Case 2, we obtain $T^{\mathbf{e}_j}(z) \in \Delta$.\\
Thus $\Delta$ is $\mathcal{T}$-invariant. Since $\Delta \subset Z_{KM, 3\eta}$, Lemma \ref{lemma3.2} implies that for every $\nu \in \mathcal{M}(\Delta ,\mathcal{T})$,
$$D(\nu, \mu_0) \le 3\eta < \eta_0.$$
Hence $\Delta$ satisfies conclusion (\ref{pro 4.1(1)}) of Proposition \ref{proposition 4.1}.
\end{proof}
\subsection{Lower bound for entropy}\label{entropy lower 4.2}
In this subsection, we present the proof of the lower bound for the entropy.\\
For $n\in \mathbb{N}$, we define a cylinder of order $n$ in $\Gamma$ as
$$C^\Gamma_{\mathbf{p}} = \{\mathscr{C} \in \Gamma : x_{\mathbf{i}}(\mathscr{C})=p_{\mathbf{i}}~\text{for~each~}\mathbf{i} \in \Lambda_{n}\},$$  
 and a cylinder of order $n$ in $\Sigma$ as
 $$C^\Sigma_{\mathbf{w}} = \{\xi \in \Sigma : \xi_{\mathbf{i}}=w_{\mathbf{i}}~\text{for~each~}\mathbf{i} \in \Lambda_{n}\}.$$
Let $\mathcal{K}_n^\Gamma$ and $\mathcal{K}_n^\Sigma$ denote the collections of all such cylinders, respectively. Then 
\begin{align}\label{4.9}
|\mathcal{K}_{n}^\Gamma| = |\Gamma_M|^{V_{n}}~~\text{and}~~|\mathcal{K}_{n}^\Sigma|=V_M^{V_{n}}.
\end{align}
Denote
$$\mathcal{K}^\Gamma:=\bigcup_{n=1}^{\infty} \mathcal{K}_{n}^\Gamma~~\text{and}~~\mathcal{K}^\Sigma:=\bigcup_{n=1}^{\infty} \mathcal{K}_{n}^\Sigma.$$
For each cylinder  $C^\Gamma \in \mathcal{K}_n^\Gamma$ and $C^\Sigma \in \mathcal{K}_n^\Sigma$, define
$$Y_{C^\Gamma, C^\Sigma} = \bigcup_{\mathscr{C} \in C^\Gamma,\; \xi \in C^\Sigma} Y_{\mathscr{C}, \xi}.$$		
\begin{lemma}\label{lemma4.5}
Let $C_1^\Gamma, C_2^\Gamma \in \mathcal{K}_n^\Gamma$ be two distinct cylinder sets and let $C^\Sigma \in \mathcal{K}_n^\Sigma$ be a fixed cylinder set. If $y_1 \in Y_{C_1^\Gamma, C^\Sigma}$ and $y_2 \in Y_{C_2^\Gamma, C^\Sigma}$, then $y_1$ and $y_2$ are $(\Lambda_{n D_M + M}, \varepsilon)$-separated.
\end{lemma}
\begin{proof}
Since $C_1^\Gamma \neq C_2^\Gamma$, there exist $\mathbf{i}_0 \in \Lambda_{n}$, $\mathscr{C}\in C_1^\Gamma$ and $\mathscr{C}'\in C_2^\Gamma$ such that $x_{\mathbf{i}_0}^{(1)}(\mathscr{C}) \neq x_{\mathbf{i}_0}^{(2)}(\mathscr{C}')$. As $C^\Sigma$ is fixed, we can choose $\xi\in C^\Sigma$ with $\xi_{\mathbf{i}_0}^{(1)} = \xi_{\mathbf{i}_0}^{(2)}:=\xi_{\mathbf{i}_0}$. 
By the definition of $Y_{\mathscr{C},\xi}$ and $Y_{\mathscr{C}',\xi}$, there exist subsets $\Lambda_1^{\mathbf{i}_0}, \Lambda_2^{\mathbf{i}_0} \subset \Lambda_M$ such that $|\Lambda_M \setminus \Lambda_1^{\mathbf{i}_0}| \le \delta V_M$, $|\Lambda_M \setminus \Lambda_2^{\mathbf{i}_0}| \le \delta V_M$. For every $\mathbf{k} \in \Lambda_1^{\mathbf{i}_0}$ and $\mathbf{k} \in \Lambda_2^{\mathbf{i}_0}$, we have
$$d\bigl(T^{\mathbf{i}_0 D_M + \xi_{\mathbf{i}_0} + \mathbf{k}}(y_1),\; T^{\mathbf{k}} x_{\mathbf{i}_0}^{(1)}(\mathscr{C})\bigr) \le \varepsilon,~~d\bigl(T^{\mathbf{i}_0 D_M + \xi_{\mathbf{i}_0} + \mathbf{k}}(y_2),\; T^{\mathbf{k}} x_{\mathbf{i}_0}^{(2)}(\mathscr{C}')\bigr) \le \varepsilon.$$
Let $A = \Lambda^{\mathbf{i}_0}_1 \cap \Lambda^{\mathbf{i}_0}_2$. By \eqref{4.4}, we have 
$$|A| \ge (1-2\delta) V_M > (1-\delta_0) V_M.$$
Since $x_{\mathbf{i}_0}^{(1)}(\mathscr{C})$ and $x_{\mathbf{i}_0}^{(2)}(\mathscr{C}')$ are distinct elements of $\Gamma$ and $\Gamma_M$ is an $(M, \delta_0, \gamma_0)$-separated, there exists $\mathbf{k}_0 \in A$ such that
$$d\bigl(T^{\mathbf{k}_0}( x_{\mathbf{i}_0}^{(1)}(\mathscr{C})),\; T^{\mathbf{k}_0} (x_{\mathbf{i}_0}^{(2)}(\mathscr{C}'))\bigr) > \gamma_0 > 3\varepsilon.$$
Thus, we have
\[\begin{aligned}
d\bigl(T^{\mathbf{i}_0 D_M +\xi_{\mathbf{i}_0}+ \mathbf{k}_0} (y_1),\; T^{\mathbf{i}_0 D_M + \xi_{\mathbf{i}_0}+ \mathbf{k}_0} (y_2)\bigr) 
&\ge d\bigl(T^{\mathbf{k}_0}( x_{\mathbf{i}_0}^{(1)}(\mathscr{C} )),\; T^{\mathbf{k}_0} (x_{\mathbf{i}_0}^{(2)}\bigr(\mathscr{C}'))) \\
&\quad - d\bigl(T^{\mathbf{i}_0 D_M + \xi_{\mathbf{i}_0} + \mathbf{k}_0} (y_1),\; T^{\mathbf{k}_0} (x_{\mathbf{i}_0}^{(1)}(\mathscr{C}))\bigr) \\
&\quad - d\bigl(T^{\mathbf{i}_0 D_M +\xi_{\mathbf{i}_0} + \mathbf{k}_0} (y_2),\; T^{\mathbf{k}_0}( x_{\mathbf{i}_0}^{(2)}(\mathscr{C}'))\bigr) \\
&> 3\varepsilon - \varepsilon - \varepsilon = \varepsilon.
\end{aligned} \]
Let $\mathbf{t} = \mathbf{i}_0 D_M + \xi_{\mathbf{i}_0} + \mathbf{k}_0$. Since $\mathbf{i}_0 \in \Lambda_n$ and $\xi_{\mathbf{i}_0} +\mathbf{k}_0 \in \Lambda_M$, we have $\mathbf{t} \in \Lambda_{n D_M + M}$. Hence $y_1$ and $y_2$ are $(\Lambda_{n D_M + M}, \varepsilon)$-separated.
\end{proof}
We now prove the lower bound for the entropy.
\begin{proof}
By the approximate $L$-product property, for each cylinder $C^\Gamma \in \mathcal{K}_n^\Gamma$, there exists some $C^\Sigma \in\mathcal{K}_n^\Sigma$ such that $Y_{C^\Gamma, C^\Sigma} \neq \emptyset$. Due to (\ref{4.7}) and (\ref{4.9}), there exists $C^\Sigma_* \in \mathcal{K}_n^\Sigma$ such that
		$$\left|\{ C^\Gamma \in \mathcal{K}_n^\Gamma : Y_{C^\Gamma, C^\Sigma_*} \neq \emptyset \}\right| \geq \frac{|\mathcal{K}_n^\Gamma|}{|\mathcal{K}_n^\Sigma|} = \frac{|\Gamma_M|^{V_n}}{V_M^{V_n}}\geq \frac{e^{V_n V_M h_1}}{V_M^{V_n}}.$$
For each such $C^\Gamma$, choose a point $y(C^\Gamma) \in Y_{C^\Gamma, C^\Sigma_*}$, thus there exists a set $E_n \subset Y$ satisfying $|E_n| \geq \frac{e^{V_n V_M h_1}}{V_M^{V_n}}$. By Lemma \ref{lemma4.5}, $E_n$ is $(\Lambda_{n D_M + M}, \varepsilon)$-separated and hence
$$s(\Delta , \Lambda_{n D_M + M}, \varepsilon) \geq |E_n| \geq \frac{e^{V_n V_M h_1}}{V_M^{V_n}}.$$
Let $N_n = n D_M + M$. By (\ref{2.1}), we have $V_{N_n}=V_{n * M}=V_nV_M$. Consequently,
$$\frac{\ln s(\Delta , \Lambda_{N_n}, \varepsilon)}{V_{N_n}} \geq \frac{V_n (V_M h_1 - \ln V_M)}{V_n V_M} = h_1 - \frac{\ln V_M}{V_M}.$$
Therefore,
$$h(\Delta,\mathcal{T}, \varepsilon) = \limsup_{n \to \infty} \frac{\ln s(\Delta , \Lambda_{N_n}, \varepsilon)}{V_{N_n}} \geq h_1 - \frac{\ln V_M}{V_M}.$$
By \eqref{4.6} $\frac{\ln V_M}{V_M} < \beta$  and $h_1 = h_0 + 10\beta$, we have
$$h(\Delta , \mathcal{T}) \geq h( \Delta, \mathcal{T}, \varepsilon) > h_1 - \beta = h_0 + 9\beta > h_0.$$
\end{proof}
\subsection{Upper bound for entropy}\label{upper bound 4.3}
In this subsection, we prove the upper bound for the entropy.
\begin{lemma}\label{lemma 4.6}
For $n\in \mathbb{N}$, let $C^\Gamma \in \mathcal{K}_n^\Gamma$ and  $C^\Sigma \in \mathcal{K}_n^\Sigma$ be any cylinders. Then
$$s(Y_{C^\Gamma, C^\Sigma}, \Lambda_{n D_M}, 2\varepsilon) \leq \left( Q(M, \delta) r(\varepsilon)^{\delta V_M} \right)^{V_n}.$$
\end{lemma}
\begin{proof}
 Let $S_1$ be a fixed $(1, \varepsilon)$-spanning set of $X$ with the minimal cardinality $r(\varepsilon)$. For a family $\mathscr{A} = (A_{\mathbf{i}})_{\mathbf{i} \in \Lambda_n}$ with each $A_{\mathbf{i}} \in \Upsilon$, define
$$Y_{C^\Gamma, C^\Sigma}(\mathscr{A}) = \left\{ y \in Y_{C^\Gamma, C^\Sigma} : \ d\bigl( T^{\mathbf{i} D_M + \xi_{\mathbf{i}} + \mathbf{k}} y,\; T^{\mathbf{k}} p_{\mathbf{i}} \bigr) \le \varepsilon~\text{for~~every }\mathbf{i} \in \Lambda_n, \mathbf{k} \in A_{\mathbf{i}} \right\}.$$
For each $\mathbf{j} \in \Lambda_{n D_M}$, there is a unique $\mathbf{i} \in\Lambda_{n}$ such that $\mathbf{j} \in \mathbf{i} D_M + \Lambda_M$. Let $\mathbf{j} = \mathbf{i} D_M + \mathbf{t}$ with $\mathbf{t} \in \Lambda_M$. Define 
\begin{align*}
\Omega_{\mathbf{j}}(\mathscr{A}):=
\begin{cases} 
\{ T^{\mathbf{k}} p_{\mathbf{i}}\}, & \text{if}~\mathbf{t} = \xi_{\mathbf{i}} + \mathbf{k}~\text{and}~\mathbf{k} \in A_{\mathbf{i}}; \\
S_1, & \text{if}~\mathbf{t}\in\Lambda_M\setminus A_{\mathbf{i}} .
\end{cases}   
\end{align*}
and denote
$$\Omega(\mathscr{A}) = \prod_{\mathbf{j} \in \Lambda_{n D_M}} \Omega_{\mathbf{j}}(\mathscr{A}).$$
Let $S_{\mathscr{A}}$ be an $(\Lambda_{n D_M}, 2\varepsilon)$-separated set in $Y_{C^\Gamma, C^\Sigma}(\mathscr{A})$. For each $y \in S_{\mathscr{A}}$ and each $\mathbf{j} \in \Lambda_{nD_M}$, we choose $\pi_{\mathbf{j}}(y) \in \Omega_{\mathbf{j}}(\mathscr{A})$ such that $d(T^{\mathbf{j}}y, \pi_{\mathbf{j}}(y)) \le \varepsilon$. This is always possible:\\
If $\mathbf{t} = \xi_{\mathbf{i}}+\mathbf{k}$ with $\mathbf{k}\in A_{\mathbf{i}}$, then by definition of $Y_{C^\Gamma, C^\Sigma}(\mathscr{A})$ we have $d(T^{\mathbf{j}}y, T^{\mathbf{k}}p_{\mathbf{i}})\le\varepsilon$, so we take $\pi_{\mathbf{j}}(y)=T^{\mathbf{k}}p_{\mathbf{i}}$; Otherwise, as $S_1$ is a $(1,\varepsilon)$-spanning set, there exists $s\in S_1$ with $d(T^{\mathbf{j}}y,s)\le\varepsilon$, and we set $\pi_{\mathbf{j}}(y)=s$.\\
Let $$\pi(y):=\{\pi_{\mathbf{j}}(y)\}_{\mathbf{j} \in \Lambda_{n D_M}}.$$
The fact that $S_{\mathscr{A}}$ is $(\Lambda_{n D_M}, 2\varepsilon)$-separated implies that $\pi: S_{\mathscr{A}} \to \Omega(\mathscr{A})$ is an injection. It follows that
$$|S_{\mathscr{A}}| \leq |\Omega(\mathscr{A})| = \prod_{\mathbf{i} \in \Lambda_n } \left( 1^{|A_{\mathbf{i}}|} r(\varepsilon)^{V_M - |A_{\mathbf{i}}|} \right) \leq \prod_{\mathbf{i} \in \Lambda_n} r(\varepsilon)^{\delta V_M} = r(\varepsilon)^{\delta V_M V_n}.$$
By the definition of $Y$, each $y \in Y_{C^\Gamma, C^\Sigma}$ must belong to some $Y_{C^\Gamma, C^\Sigma}(\mathscr{A})$, hence
$$Y_{C^\Gamma, C^\Sigma} = \bigcup_{\mathscr{A} \in \Upsilon^{\Lambda_n}} Y_{C^\Gamma, C^\Sigma}(\mathscr{A}).$$
If $E$ is any $(\Lambda_{n D_M}, 2\varepsilon)$-separated subset of $Y_{C^\Gamma, C^\Sigma}$, then $E = \bigcup_{\mathscr{A}} (E\cap Y_{C^\Gamma, C^\Sigma}(\mathscr{A}))$. For each $\mathscr{A}$, the set $E\cap Y_{C^\Gamma, C^\Sigma}(\mathscr{A})$ is also $(\Lambda_{n D_M}, 2\varepsilon)$-separated, so its cardinality is at most $|S_{\mathscr{A}}|$. Hence
$$s(Y_{C^\Gamma, C^\Sigma}, \Lambda_{n D_M}, 2\varepsilon) \leq \sum_{\mathscr{A} \in \Upsilon^{\Lambda_n}} |S_{\mathscr{A}}| \leq |\Upsilon|^{V_n} r(\varepsilon)^{\delta V_M V_n} = \left( Q(M, \delta) r(\varepsilon)^{\delta V_M} \right)^{V_n}.$$
\end{proof}
\begin{lemma}\label{lemma4.7}
For $n\in \mathbb{N}$, we have
 $$ s(\Delta , \Lambda_{n D_M}, 2\varepsilon) < V_M  \left( e^{V_M (h_1 + \beta)} V_M  Q(M, \delta) r(\varepsilon)^{\delta V_M} \right)^{V_{n+1}}. $$
\end{lemma}
\begin{proof}
From \eqref{4.7} and Lemma \ref{lemma 4.6}, we obtain
\begin{align*}
  s(Y, \Lambda_{n D_M}, 2\varepsilon)&\le \sum_{C^\Gamma \in \mathcal{K}_n^\Gamma,C^\Sigma \in \mathcal{K}_n^\Sigma}s(Y_{C^\Gamma, C^\Sigma},  \Lambda_{n D_M}, 2\varepsilon)\\&\le |\Gamma_M|^{V_{n}}V_M^{V_{n}}\left( Q(M, \delta) r(\varepsilon)^{\delta V_M} \right)^{V_n}\\&<\Bigl(e^{V_M (h_1 + \beta)}V_MQ(M, \delta) r(\varepsilon)^{\delta V_M}\Bigr)^{V_n}. 
\end{align*}
Note that if $S$ is an $(\Lambda_{n D_M},2\varepsilon)$-separated subset of $T^{\mathbf{k}}(Y)$, then $T^{-\mathbf{k}}(S)$ contains an $(\Lambda_{n D_M}+\mathbf{k},2\varepsilon)$-separated subset of $Y$. Since $\Lambda_{n D_M}+\mathbf{k} \subset \Lambda_{(n+1)D_M}$ for any $\mathbf{k}\in\Lambda_M$, we have
$$s(T^{\mathbf{k}}(Y),\Lambda_{n D_M},2\varepsilon)\le s(Y,\Lambda_{n D_M}+\mathbf{k},2\varepsilon)\le s(Y,\Lambda_{(n+1)D_{M}},2\varepsilon).$$
It follows from (\ref{4.8}) that
\begin{align*}
  s(\Delta , \Lambda_{n D_M}, 2\varepsilon)&\le \sum_{\mathbf{k} \in \Lambda_{M}}s(T^{\mathbf{k}}Y, \Lambda_{n D_M}, 2\varepsilon)\\&\le V_{M}s(Y,\Lambda_{(n+1)D_{M}},2\varepsilon)\\&<V_{M}\Bigl(e^{V_M (h_1 + \beta)}V_MQ(M, \delta) r(\varepsilon)^{\delta V_M}\Bigr)^{V_{n+1}}. 
\end{align*} 
\end{proof}      
We now establish the upper bound for the entropy.		
\begin{proof}
For any $t \in \mathbb{N}$, choose $n=n(t)\in \mathbb{N}$ such that $(n-1)D_M \le t < nD_M$. Then $\Lambda_t \subset \Lambda_{n D_M}$ and $V_t \ge V_{(n-1)D_M}$. By Lemma \ref{lemma4.7}, we have
\begin{align*}
 h(\Delta , \mathcal{T}, 2\varepsilon)& = \limsup_{t \to \infty} \frac{\ln s(\Delta , t, 2\varepsilon)}{V_t}\\& \le\limsup_{t \to \infty} \frac{\ln s(\Delta , \Lambda_{n D_M}, 2\varepsilon)}{V_{(n-1)D_M}}\\&\le  \frac{\ln V_M}{V_{(n-1)D_M}} + \frac{V_{n+1}}{V_{(n-1)D_M}} \left( V_M (h_1 + \beta) + \ln V_M + \ln Q(M, \delta) + \delta V_M \ln r(\varepsilon) \right).
\end{align*}
As $n \to \infty$, $\frac{\ln V_M}{V_{(n-1)D_M}}\to 0$ and $\frac{V_{n+1}}{V_{(n-1)D_M}} \to \frac{1}{V_M}$. Hence    
\begin{align*}
 h(\Delta ,\mathcal{T}, 2\varepsilon) \le \frac{1}{V_M} \left( V_M (h_1 + \beta) + \ln V_M + \ln Q(M, \delta) + \delta V_M \ln r(\varepsilon) \right).   
\end{align*}
By (\ref{4.4})-(\ref{4.6}), we have
	\begin{itemize}
		\item \(\frac{\ln V_M}{V_M} < \beta\);
		\item \(\frac{\ln Q(M, \delta)}{V_M} \le -\delta \ln \delta - (1-\delta)\ln(1-\delta) < \beta\);
		\item \(\delta \ln r(\varepsilon) < \beta\).
\end{itemize}
Therefore,
$$h(\Delta ,\mathcal{T}, 2\varepsilon) <  h_1 + 4\beta< h_0 + \beta_0.$$
\end{proof}	
When $h_{\mu_0}(\mathcal{T}) = 0$, we have the following result.
\begin{proposition}\label{pro4.8}
Let $(X, d,\mathcal{T})$ be a system with the approximate $L$-product property. Assume that $\mu_0 \in \mathcal{M}_e(X, \mathcal{T})$, $h_{\mu_0}(\mathcal{T}) = 0$ and $\eta_0 > 0$. Then there exists a compact $\mathcal{T}$-invariant subset $\Delta = \Delta (\mu_0, \eta_0)$ such that $D(\nu, \mu_0) < \eta_0$ for every $\nu \in \mathcal{M}(\Delta , \mathcal{T})$.
\end{proposition}
\subsection{Proof of Theorem \ref{theorem 1.3}}
In this subsection, we present the proof of Theorem \ref{theorem 1.3}.
\begin{proof} 
Take $\mu \in \mathcal{M}(X, \mathcal{T})$ and let $U$ be a neighborhood of $\mu$. Choose $h \in (0, h_\mu(\mathcal{T}))$ and $\varepsilon, \beta > 0$. There exists $\eta_0 > 0$ such that $B(\mu, 2\eta_0) \subset U$. Proposition \ref{proposition2.15}guarantees the existence of an ergodic measure $\mu_0 \in \mathcal{M}_{e}(X, \mathcal{T})$ with
$$D(\mu, \mu_0) < \eta_0~~\text{and}~~h_{\mu_0}(\mathcal{T}) > h.$$
Applying Proposition \ref{proposition 4.1}, we obtain $\gamma \in (0, \varepsilon)$ and a compact $\mathcal{T}$-invariant set $\Delta $ such that $D(\nu, \mu_0) < \eta_0$ for every $\nu \in \mathcal{M}(\Delta, \mathcal{T})$, $h(\Delta , \mathcal{T}) > h$ and $h(\Delta , \mathcal{T}, \gamma) < h + \beta$. Consequently,
$$\mathcal{M}(\Delta , \mathcal{T}) \subset B(\mu, 2\eta_0) \subset U.$$
If $h_\mu(\mathcal{T}) = 0$, the same conclusion follows by an analogous argument using Proposition \ref{pro4.8}.
\end{proof}	
\subsection{Minimal systems}\label{subsection 4.5}
In this subsection, we prove Corollary \ref{cor1.6}. The proof is divided into two parts, presented as Corollaries \ref{cor 4.9} and \ref{cor 4.10}, both of which are direct consequences of Proposition \ref{proposition 4.1}.
\begin{corollary}\label{cor 4.9}
Let $(X, d, \mathcal{T})$ be a topological dynamical system with the approximate $L$-product property and positive topological entropy. Then $(X, d, \mathcal{T})$ is not minimal.
\end{corollary}
\begin{proof}
Assume that $h(\mathcal{T}) > 0$. By \eqref{2.2}, we can choose $\varepsilon_{0} > 0$ such that $h(X, \mathcal{T}, \varepsilon_{0}) > 0$. The variational principle then yields an ergodic measure $\mu_{0} \in \mathcal{M}_{e}(X, \mathcal{T})$ with $h_{\mu_{0}}(\mathcal{T}) > 0$. Choose $h_{0} \in (0, h_{\mu_{0}}(\mathcal{T}))$ and $\beta_{0} > 0$ satisfying
$$0 < h_{0} + \beta_{0} < h(X, \mathcal{T}, \varepsilon_{0}).$$
Applying Proposition \ref{proposition 4.1}, we obtain $\gamma \in (0, \varepsilon_{0})$ and a compact $\mathcal{T}$-invariant subset $\Delta \subset X$ such that
$$h(\Delta, \mathcal{T}, \gamma) < h_{0} + \beta_{0} < h(X, \mathcal{T}, \varepsilon_{0}) \leq h(X, \mathcal{T}, \gamma).$$
The strict inequality $h(\Delta , \mathcal{T}, \gamma) < h(X, \mathcal{T}, \gamma)$ shows that $\Delta$ is a proper subset of $X$. Since $\Delta$ is a nonempty proper compact and $\mathcal{T}$-invariant subset of $X$, this contradicts the definition of minimality. Hence, $(X, d, \mathcal{T})$ is not minimal.
\end{proof}
\begin{corollary}\label{cor 4.10}
Let $(X, d, \mathcal{T})$ be a topological dynamical system with the approximate $L$-product property that is not uniquely ergodic. Then $(X, d, \mathcal{T})$ is not minimal.
\end{corollary}
\begin{proof}
Let $\mu_{1}$ and $\mu_{2}$ be two distinct ergodic measures and choose $0 < \eta_{0} < \frac{1}{3} D(\mu_{1}, \mu_{2})$. By Proposition \ref{proposition 4.1} and Proposition \ref{pro4.8}, we can find compact $\mathcal{T}$-invariant subsets $\Delta_{1}$ and $\Delta_{2}$ such that:
\begin{itemize}
     \item For every $\nu \in \mathcal{M}(\Delta_1,\mathcal{T})$, we have $D(\nu,\mu_1) \le \eta_0$.
    \item For every $\nu \in \mathcal{M}(\Delta_2,\mathcal{T})$, we have $D(\nu,\mu_2) \le \eta_0$.
\end{itemize}
If $\Delta_1$ and $\Delta_2$ intersected, then their intersection would be a nonempty compact $\mathcal{T}$-invariant set and would support an invariant measure $\nu$. This $\nu$ would belong to both $\mathcal{M}(\Delta_1,\mathcal{T})$ and $\mathcal{M}(\Delta_2,\mathcal{T})$, leading to
$$D(\mu_1,\mu_2) \le D(\mu_1,\nu) + D(\nu,\mu_2) \le 2\eta_0 < \frac{2}{3}D(\mu_1,\mu_2) < D(\mu_1,\mu_2),$$
a contradiction. Hence $\Delta_1$ and $\Delta_2$ are disjoint. Consequently, both $\Delta_1$ and $\Delta_2$ are nonempty proper compact and $\mathcal{T}$-invariant subsets of $X$, which proves that $(X, d, \mathcal{T})$ is not minimal.
\end{proof}
\section{Intermediate entropy}
In this section, we prove Theorem \ref{theorem 1.4} and Corollary \ref{cor1.5}. Theorem \ref{theorem 1.4} (\ref{theorem 1.4(1)}) is an immediate corollary of Theorem \ref{theorem 1.3} and Proposition \ref{proposition 5.1}.
\begin{proposition}\label{proposition 5.1}
Let $(X, d, \mathcal{T})$ be an asymptotically entropy expansive system and let $\mu \in \mathcal{M}(X, \mathcal{T})$ be almost entropy-approximable. Then $\mu$ is entropy-approximable.
\end{proposition}
\begin{proof}
Consider any neighborhood $U$ of $\mu$ in $\mathcal{M}(X, \mathcal{T})$, $h \in (0, h_\mu(\mathcal{T}))$ and $\beta > 0$. By Definition \ref{definition 2.8}(3) of asymptotically entropy expansiveness, we can choose $\varepsilon_0 > 0$ such that for every $\varepsilon' \in (0, \varepsilon_0)$,
\begin{align}\label{5.1}
h^*(\mathcal{T}, \varepsilon') < \frac{\beta}{2}.
\end{align}
Since $\mu$ is almost entropy-approximable, there exist a compact $\mathcal{T}$-invariant set $\Delta$ and a number $\gamma \in (0, \varepsilon_0)$ satisfying
\begin{align}\label{5.2}
 \mathcal{M}(\Delta , \mathcal{T}) \subset U,~~h(\Delta, \mathcal{T}) > h~~ \text{and}~~h(\Delta , \mathcal{T}, \gamma) < h + \frac{\beta}{2}.   
\end{align}
Combining Proposition \ref{proposition 2.9} with \eqref{5.1} and \eqref{5.2}, we obtain
\begin{align*}
h(\Delta , \mathcal{T}) \leq h(\Delta , \mathcal{T}, \gamma) + h^*(\mathcal{T}, \gamma) < h + \frac{\beta}{2} + \frac{\beta}{2} = h + \beta. 
\end{align*}
Therefore $\mu$ is entropy-approximable.
\end{proof}
According to Proposition \ref{proposition 2.9}, asymptotic entropy expansiveness implies the upper semi-continuity of the entropy map. Consequently, Theorem \ref{theorem 1.4} (\ref{theorem 1.4(2)}) follows from Theorem \ref{theorem 1.4} (\ref{theorem 1.4(1)}) and Proposition \ref{proposition 5.2}.
\begin{proposition}\label{proposition 5.2}
Let $(X, d, \mathcal{T})$ be a topological dynamical system. Assume that the following two conditions hold:
\begin{enumerate}
    \item Every invariant measure $\mu \in \mathcal{M}(X, \mathcal{T})$ is entropy-approximable.
    \item  The entropy map $\mu \mapsto h_\mu(\mathcal{T})$ is upper semi-continuous on $\mathcal{M}(X, \mathcal{T})$.
\end{enumerate}
Then the system $(X, d, \mathcal{T})$ is \emph{entropy-generic}. 
\end{proposition}
\begin{proof}
We present the proof in three steps. \\
\textbf{Step 1: $\mathcal{M}^\alpha(X, \mathcal{T})$ is a Baire space.}\\
The upper semicontinuity of the entropy map implies that $\mathcal{M}^\alpha(X,\mathcal{T}) := \{ \mu \in \mathcal{M}(X,\mathcal{T}) : h_\mu(\mathcal{T}) \ge \alpha \}$ is a closed subset of $\mathcal{M}(X,\mathcal{T})$. Consequently, $\mathcal{M}^\alpha(X,\mathcal{T})$ is compact and metrizable, and thus completely metrizable. It follows that $\mathcal{M}^\alpha(X,\mathcal{T})$ is a Baire space.\\
\textbf{Step 2: Constructing a dense $G_\delta$ set.}\\
Fix $\alpha \in [0, h(\mathcal{T}))$ and choose $\alpha'$ with $\alpha < \alpha' < h(\mathcal{T})$. Define
$$\mathcal{M}(\alpha, \alpha') := \{ \mu \in \mathcal{M}_e(X, \mathcal{T}) : \alpha \leq h_\mu(\mathcal{T}) < \alpha' \}.$$
\noindent \textit{(1) $\mathcal{M}(\alpha,\alpha')$ is a $G_\delta$ set in $\mathcal{M}^\alpha(X,\mathcal{T})$.}    
 By upper semi-continuity of the entropy map, the set $O := \{ \mu \in \mathcal{M}(X, \mathcal{T}) : h_\mu(\mathcal{T}) < \alpha' \}$ is open. Hence $O \cap \mathcal{M}^\alpha(X,\mathcal{T})$ is relatively open in the subspace $\mathcal{M}^\alpha(X,\mathcal{T})$. Moreover, $\mathcal{M}_e(X, \mathcal{T})$ is a $G_\delta$ subset. Therefore,
    $$\mathcal{M}(\alpha, \alpha') = \mathcal{M}_e(X, \mathcal{T}) \cap O \cap \mathcal{M}^\alpha(X, \mathcal{T})$$
is a $G_\delta$ set in the subspace $\mathcal{M}^\alpha(X,\mathcal{T})$.\\
\noindent \textit{(2) $\mathcal{M}(\alpha,\alpha')$ is dense in $\mathcal{M}^\alpha(X,\mathcal{T})$.}
Let $\mu \in \mathcal{M}^\alpha(X, \mathcal{T})$ and $\eta > 0$. We need to find $\nu \in \mathcal{M}(\alpha, \alpha')$ with $D(\nu, \mu) < \eta$. Since $\alpha' < h(\mathcal{T})$, the variational principle yields an ergodic measure $\mu_X \in \mathcal{M}_e(X, \mathcal{T})$ such that $h_{\mu_X}(\mathcal{T}) = h(\mathcal{T}) >\alpha'$. Recall that $D^*$ is the diameter of $\mathcal{M}(X)$, let $\theta =\frac{\eta}{3D^*}$ and define
$$\mu' = (1 - \theta) \mu + \theta \mu_X.$$
By the convexity property of the metric $D$, we have $D(\mu', \mu) \le \theta D^* = \frac{\eta}{3}$. The affinity of the entropy map gives
$$h_{\mu'}(\mathcal{T}) = (1 - \theta) h_\mu(\mathcal{T}) + \theta h_{\mu_X}(\mathcal{T}) > (1 - \theta) \alpha + \theta \alpha' > \alpha.$$
Because $\mu'$ is entropy-approximable and $\alpha \in [0, h_{\mu'}(\mathcal{T}))$, we can choose $\beta$ with
$$0 < \beta < \min\{\alpha' - \alpha,\; h_{\mu'}(\mathcal{T}) - \alpha\}.$$
Then there exists a compact $\mathcal{T}$-invariant set $\Delta \subset X$ such that
$$\mathcal{M}(\Delta, \mathcal{T}) \subset B\left(\mu', \frac{\eta}{3}\right)~~\text{and}~~\alpha < h(\Delta, \mathcal{T}) < \alpha + \beta < \alpha'.$$
By the variational principle applied to the subsystem $(\Delta, \mathcal{T})$, we can choose $\nu \in \mathcal{M}_e(\Delta , \mathcal{T})\subset B(\mu, \eta)$. It follows that $h_\nu(\mathcal{T}) \in [\alpha, \alpha')$, so $\nu \in \mathcal{M}(\alpha, \alpha')$. Moreover,
$$D(\nu,\mu) \le D(\nu,\mu') + D(\mu',\mu) < \frac{\eta}{3} + \frac{\eta}{3} = \frac{2\eta}{3} < \eta.$$
Thus $\mathcal{M}(\alpha, \alpha')\cap B(\mu, \eta)\neq \emptyset$, proving that $\mathcal{M}(\alpha, \alpha')$ is dense in $\mathcal{M}^\alpha(X, \mathcal{T})$. \\
\textbf{Step 3: Obtaining a residual set.}\\
For each $k\in \mathbb{N}$ large enough so that $\alpha_k := \alpha + \frac{1}{k} < h(\mathcal{T})$, the set $\mathcal{M}(\alpha, \alpha_k)$ is a dense $G_\delta$ set in $\mathcal{M}^\alpha(X, \mathcal{T})$ by Steps 1 and 2. Since $\mathcal{M}^\alpha(X,\mathcal{T})$ is a Baire space, the countable intersection
$$\bigcap_{k=1}^{\infty} \mathcal{M}(\alpha, \alpha + \tfrac{1}{k}) = \{ \mu \in \mathcal{M}_e(X, \mathcal{T}) : h_\mu(\mathcal{T}) = \alpha \} = \mathcal{M}_e(X, \mathcal{T}, \alpha)$$
is also a dense $G_\delta$ set, hence residual in $\mathcal{M}^\alpha(X, \mathcal{T})$. This establishes that $(X, d, \mathcal{T})$ is entropy-generic.
\end{proof}
Entropy-genericity implies that Corollary \ref{cor1.5} follows directly from Theorem \ref{theorem 1.4} (\ref{theorem 1.4(2)}) and Proposition \ref{pro5.3}.
\begin{proposition}\label{pro5.3}
Let $(X, d, \mathcal{T})$ be a topological dynamical system. If $(X, d, \mathcal{T})$ is entropy-generic, then for every $\mu\in \mathcal{M}(X, \mathcal{T})$ and every neighborhood $U$ of $\mu$, we have
\[\begin{cases} 
\mathbb{H}(X, \mathcal{T}, U) \supset [0, h_\mu(\mathcal{T})], & \text{if } h_\mu(\mathcal{T}) < h(\mathcal{T}); \\ 
\mathbb{H}(X, \mathcal{T}, U) \supset [0, h_\mu(\mathcal{T})), & \text{if } h_\mu(\mathcal{T}) = h(\mathcal{T}).
\end{cases}\]
\end{proposition}
\begin{proof}
Choose $\eta > 0$ such that $B(\mu, 2\eta) \subset U$.
Let $0 \le\alpha < h_\mu(\mathcal{T}) \le h(\mathcal{T})$.
Since $(X, d, \mathcal{T})$ is entropy-generic, the set
$$\mathcal{M}_e(X, \mathcal{T}, \alpha) := \{ \nu \in \mathcal{M}_e(X, \mathcal{T}) : h_\nu(\mathcal{T}) = \alpha \}$$
is residual in the subspace
$$\mathcal{M}^\alpha(X, \mathcal{T}) := \{ \nu \in \mathcal{M}(X, \mathcal{T}) : h_\nu(\mathcal{T}) \geq \alpha \}.$$
As $\mu \in \mathcal{M}^\alpha(X, \mathcal{T})$, the relatively open set $B(\mu, \eta) \cap \mathcal{M}^\alpha(X, \mathcal{T})$ is nonempty. By density, there exists an ergodic measure
$$\nu \in \mathcal{M}_e(X, \mathcal{T}, \alpha) \cap B(\mu, \eta) \subset U,$$
such that $h_\nu(\mathcal{T}) = \alpha$. Hence,
\begin{align}\label{5.3}
 \mathbb{H}(X, \mathcal{T}, U) \supset [0, h_\mu(\mathcal{T})).   
\end{align}
Now assume $h_\mu(\mathcal{T}) < h(\mathcal{T})$. By the variational principle, there exists $\mu_0 \in \mathcal{M}(X, \mathcal{T})$ such that $h_{\mu_0}(\mathcal{T}) > h_\mu(\mathcal{T})$. Denote
$$\mu' := \left(1 - \frac{\eta}{D^*}\right) \mu + \frac{\eta}{D^*} \mu_0.$$
A direct check shows $\mu' \in B(\mu, 2\eta) \subset U$, and
$$h_{\mu'}(\mathcal{T})=\left(1 - \frac{\eta}{D^*}\right) h_{\mu}(\mathcal{T}) + \frac{\eta}{D^*} h_{\mu_0}(\mathcal{T}) > h_{\mu}(\mathcal{T}).$$
Applying \eqref{5.3} to $\mu'$ yields  
$$\mathbb{H}(X,\mathcal{T},U) \supset [0, h_{\mu'}(\mathcal{T}))\supset [0, h_\mu(\mathcal{T})].$$
\end{proof}

\section*{Compliance with Ethical Standards}

\begin{itemize}
    \item \textbf{Funding:} The  second author was supported by the
National Natural Science Foundation of China (No.12471184).
The third authors was  supported by the
National Natural Science Foundation of China (No.11971236), Qinglan Project of Jiangsu Province of China.
    \item \textbf{Conflict of Interest:} The authors have no relevant financial or non-financial interests to disclose.
    \item \textbf{Research Involving Human Participants and/or Animals:} This research did not involve human participants or animals.
    \item \textbf{Data Availability Statement:} Data sharing is not applicable to this article as no datasets were generated or analyzed during the current study.
\end{itemize}


\begin{thebibliography}{MM}
\bibitem[BCL07]{BCL07} F. B\'eguin, S. Crovisier, F. Le Roux, Construction of curious minimal uniquely ergodic homeomorphisms on manifolds: the Denjoy-Rees technique.
 \emph{Ann. Sci. \'Ec. Norm. Sup\'er.} {\bf40}(4)(2007)251-308.
\bibitem[B71]{B71} R. Bowen, Periodic points and measures for Axiom A diffeomorphisms. \emph{Transl. Am. Math. Soc.} \textbf{154} (1971) 377-397.
 \bibitem[B20]{B20} D. Burguet, Topological and almost Borel universality for systems with the weak specification property. \emph{Ergod. Theory Dyn. Syst.}  \textbf{40} (8) (2020) 2098-2115.
 \bibitem[CM21]{CM21}
 N. Chandgotia, T. Meyerovitch, Borel subsystems and ergodic universality for compact $\mathbb{Z}^d$-systems via specification and beyond.
 \emph{Proc. Lond. Math. Soc.} \textbf{123} (2021) 231-312.
 \bibitem[CLT20]{CLT20}D. Constantine, J. Lafont, D. J. Thompson, The weak specification property for geodesic flows on CAT($-1$) spaces, \emph{Groups Geom. Dyn.} \textbf{14} (1) (2020) 297-336.
\bibitem[GW94]{GW94} E. Glasner, B. Weiss,
Strictly ergodic, uniform positive entropy models.\emph{Bull. Soc. Math. Fr. }
\textbf{122} (3) (1994) 399-412.
\bibitem[GSW17]{GSW17} L. Guan, P. Sun, W. Wu, Measures of intermediate entropies and homogeneous dynamics.\emph{Nonlinearity.}\textbf{30} (2017) 3349-3361.
\bibitem[HK67]{HK67} F. Hahn, Y. Katznelson,
On the entropy of uniquely ergodic transformations.
\emph{Trans. Am. Math. Soc.} \textbf{126} (1967) 335-360.
\bibitem[HXX21]{HXX21} W. Huang, L. Xu, S. Xu, Ergodic measures of intermediate entropy for affine transformations of nilmanifolds. \emph{Electron. Res. Arch.} \textbf{29} (4) (2021) 2819-2827.
\bibitem[K80]{K80} A. Katok, Lyapunov exponents, entropy and periodic orbits for diffeomorphisms. \emph{Publ. Math. IHES} \textbf{51} (1980) 137-173.
\bibitem[KKK18]{KKK18} J. Konieczny, M. Kupsa, D. Kwietniak,
Arcwise connectedness of the set of ergodic measures of hereditary shifts.
\emph{Proc. Am. Math. Soc.} \textbf{146} (8) (2018) 3425-3438.
\bibitem[KLO16]{KLO16} D. Kwietniak, M. Lacka, P. Oprocha, A panorama of specification-like properties and their consequences. \emph{Contemp. Math.} \textbf{669} (2016) 155-186.
\bibitem[LO18]{LO18} J. Li, P. Oprocha, Properties of invariant measures in dynamical systems with the shadowing property. \emph{Ergod. Theory Dyn. Syst.} \textbf{38} (2018) 2257-2294.
\bibitem[LOS78]{LOS78}J. Lindenstrauss, G. Olsen, Y.Sternfeld, The Poulsen simplex. \emph{Ann. Inst. Fourier (Grenoble).} {\bf28}(1),vi, 91-114 (1978).
\bibitem[PS05]{PS05}C.-E. Pfister, W.G. Sullivan, Large deviations estimates for dynamical systems without the specification property. Application to the $\beta$-shifts. \emph{ Nonlinearity.} {\bf18}(2005) 237–261.
\bibitem[P01]{P01}R.R. Phelps, Lectures on Choquet’s Theorem, second ed., Lecture Notes in Mathematics, vol. 1757, \emph{Springer-Verlag, Berlin}, 2001.
\bibitem[QS16]{QS16} A. Quas, T. Soo, Ergodic universality of some topological dynamical systems. \emph{Trans. Am. Math. Soc. } \textbf{368} (6) (2016) 4137-4170.
\bibitem[RS16]{RS16} X. Ren, W. Sun: Local entropy, metric entropy and topological entropy for countable discrete amenable group actions. \emph{Int. J. Bifurc. Chaos.} {\bf26}(07) (2016) 1650110.
\bibitem[S10a]{S10a}P. Sun, Zero-entropy invariant measures for skew product diffeomorphisms.\emph{Ergod. Theory Dyn. Syst.} \textbf{30} (2010) 923-930.
\bibitem[S10b]{S10b}P. Sun, Measures of intermediate entropies for skew product diffeomorphisms.\emph{Discrete Contin. Dyn. Syst., Ser.A } \textbf{27} (3) (2010) 1219-1231.
\bibitem[S12]{S12}P. Sun, Density of metric entropies for linear toral automorphisms.
\emph{Dyn. Syst.} \textbf{27} (2) (2012) 197-204.
\bibitem[S21]{S21}P. Sun, Equilibrium states of intermediate entropies. \emph{Dyn. Syst.} \textbf{36} (1) (2021) 69-78.
\bibitem[S25]{S25} P. Sun, Ergodic measures of intermediate entropies for dynamical systems with approximate product property. \emph{Adv. Math.} {\bf465} (8) (2025) 110159.
\bibitem[U12]{U12} R. Ures, Intrinsic ergodicity of partially hyperbolic diffeomorphisms with a hyperbolic linear part. \emph{Proc. Am. Math. Soc.}\textbf{140} (6) (2012) 1973-1985.
\bibitem[W82]{W82}P. Walters, An Introduction to Ergodic Theory. \emph{Springer-Verlag}, 1982.
\bibitem[YZ20]{YZ20} D. Yang, J. Zhang, Non-hyperbolic ergodic measures and horseshoes in partially hyperbo lic homoclinic classes.
\emph{J. Inst. Math. Jussieu.} \textbf{19} (5) (2020) 1765-1792.
\end{thebibliography}
\end{document}